\documentclass[lettersize,journal]{IEEEtran}
\usepackage{amsmath,amsfonts,amsthm}
\usepackage{algorithmic}
\usepackage{algorithm}
\usepackage{array}
\usepackage[caption=false,font=normalsize,labelfont=sf,textfont=sf]{subfig}
\usepackage{textcomp}
\usepackage{stfloats}
\usepackage{url}
\usepackage{verbatim}
\usepackage{graphicx}
\usepackage{cite}

\usepackage[dvipsnames]{xcolor}
\usepackage{multirow, multicol}

\newtheorem{theorem}{\bf Theorem}[section]
\newtheorem{definition}{\bf Definition}[section]
\newtheorem{remark}{\bf Remark}[section]
\newtheorem{lemma}{\bf Lemma}[section]
\newtheorem{example}{\bf Example}[section]

\renewcommand{\IEEEPARstartWORDCAPSTYLE}{} 
\renewcommand{\IEEEPARstartWORDFONTSTYLE}{} 

\usepackage{orcidlink} 

\begin{document}

\title{Second-Generation Wavelet-inspired Tensor Product
with Applications in Hyperspectral Imaging}

\author{%
Aneesh Panchal\,\orcidlink{0009-0005-5727-3901} and 
Ratikanta Behera\,\orcidlink{0000-0002-6237-5700}%
\thanks{The authors are in the Department of Computational and Data Sciences, Indian Institute of Science, Bangalore, Karnataka-560012, India (e-mails: aneeshp@iisc.ac.in, ratikanta@iisc.ac.in)}
}

\maketitle

\begin{abstract}
This paper introduces the $w$-product, a novel wavelet-based tensor multiplication scheme leveraging second-generation wavelet transforms to achieve linear transformation complexity while preserving essential algebraic properties. The $w$-product outperforms existing tensor multiplication approaches by enabling fast and numerically stable tensor decompositions by proposing ``$w$-svd'' and its sparse variant ``sp-$w$-svd'', for efficient low-rank approximations with significantly reduced computational costs. Experiments on low-rank hyperspectral image reconstruction demonstrate up to a $92.21$ times speedup compared to state-of-the-art ``$t$-svd'', with comparable PSNR and SSIM metrics. We discuss the Moore-Penrose inverse of tensors based on the $w$-product and examine its essential properties. Numerical examples are provided to support the theoretical results. Then, hyperspectral image deblurring experiments demonstrate up to $27.88$ times speedup with improved image quality. In particular, the $w$-product and the sp-$w$-product exhibit exponentially increasing acceleration with the decomposition level compared to the traditional approach of the $t$-product. This work provides a scalable framework for multidimensional data analysis, with future research directions including adaptive wavelet designs, higher-order tensor extensions, and real-time implementations.
\end{abstract}


\section{Introduction and Motivations}\label{sec:intro}
\IEEEPARstart{T}{ensors} serve as powerful tools for analyzing 
multidimensional data in computer vision~\cite{panagakis2021tensor}, hyperspectral imaging~\cite{fan2017hyperspectral}, and signal processing~\cite{qi2017tensor}. 
Tensor algebra provides an elegant mathematical framework for representing and analyzing multiway relationships through compact formulations 
such as $t$-product~\cite{kilmer2011, Miao2020Gene}, $\phi$-product~\cite{Songphiproduct}, and $m$-product~\cite{kernfeld2015}.
However, these frameworks often suffer from high computational costs due to the heavy use of Fast Fourier Transform (FFT)-based transformations ($t$-product) and large matrix multiplications ($\phi$-product and $m$-product), which makes them unsuitable for large-scale data~\cite{liu2024revisiting}.

To address this limitation, we propose a new wavelet-based tensor multiplication approach, referred to as the $w$-product. This approach is based on second-generation wavelets (lifting scheme wavelets) and achieves linear computational complexity in transformation while preserving the essential algebraic properties of the tensor operations. 
The $w$-product captures both global and local correlations through a recursive process that derives fine-scale detail coefficients from coarse-scale coefficients at successive resolution levels. The proposed framework exhibits superior computational efficiency because it requires significantly fewer arithmetic operations than established tensor multiplication approaches while maintaining perfect reconstruction capabilities. In addition, it provides enhanced numerical stability and supports multiresolution representation for higher-order data analysis.

Tensors have also been extensively used to represent hyperspectral data~\cite{cao2020enhanced, liu2024bi}. Currently, hyperspectral tensor analysis primarily relies on the $t$-product~\cite{xu2019nonlocal, kong2025image}, which uses FFT and thus lacks the ability to capture localized information~\cite{appiah2024performance, sharif2014comparative}. Additionally, the transformation between real and complex domains causes loss of phase information, which is critical for image analysis. The proposed $w$-product offers a novel approach for hyperspectral image analysis that captures localized features and operates exclusively in the real domain.

The major contributions of this work are:
\begin{enumerate}
\item We introduce the $w$-product, an approach that transforms tensor operations into matrix operations in the wavelet domain, where tensors decompose into sequences of independent matrices with linear transformation complexity, enabling efficient processing while preserving the inherent multidimensional characteristics.

\item We prove the fundamental algebraic properties within the proposed $w$-product framework, specifically the trace, singular value decomposition (SVD), and Moore-Penrose inverse of tensors.
    
\item We demonstrate through a detailed complexity analysis that the proposed method achieves significant computational advantages over existing formulations of $t$-products and $m$-products. 

\item We propose an efficient tensor decomposition framework called ``$w$-svd'' and its sparse variant ``sp-$w$-svd'' for the low-rank approximation. These approaches exhibit significantly improved computational efficiency compared to existing state-of-the-art tensor decomposition techniques.

\item We demonstrate the effectiveness of our proposed tensor multiplication framework through hyperspectral image reconstruction and deblurring experiments, specifically, in image reconstruction, achieving up to $92.21$ times speedup while maintaining comparable PSNR and SSIM values, and image deblurring achieved a speedup of up to $27.88$ times along with improved PSNR and SSIM values.

\end{enumerate}

The paper is organized as follows: Section~\ref{sec:defn} defines the $w$-product and its properties, Section~\ref{sec:complexity} analyzes computational complexity, Section~\ref{sec:application} presents the ``sp-$w$-svd'' algorithm and demonstrates applications of the $w$-product in low-rank hyperspectral image reconstruction and hyperspectral image deblurring. Finally, Section~\ref{sec:conclusion} concludes the article with future research directions.

\section{Definitions and Theorems}\label{sec:defn}
Let $\mathbb{R}^{n_1 \times n_2 \times p}$ be the set of third-order real-valued tensors of dimensions $n_1 \times n_2 \times p$. The $k$th frontal slice of the tensor $\mathcal{A}\in\mathbb{R}^{n_1 \times n_2 \times p}$ is denoted by $\mathcal{A}^{(k)}=\mathcal{A}(:,:,k)$.
Elements of $\mathcal{A}$ are denoted by $\mathcal{A}[{i,j,k}]$, such that $i=1,2, \dots, n_1$, $j = 1,2, \dots, n_2$ and $k = 1,2, \dots, p$.  
\subsection{Second-generation wavelet transforms}
Let $f_{(0)} = f \in \mathbb{R}^{p}$ with $p = 2^L m$, where $m, L \in \mathbb{N}$ and $L$ denotes the number of levels. The second-generation wavelet transformation~\cite{Swe98Thelift} is expressed as follows,
\begin{equation}\label{eq:2ndgen_wavelet_forward}
    W_{[L]}(f) = \left(s_{L}, d_{L}, d_{(L-1)}, \dots, d_{1}\right),
\end{equation}
with $\left(s_{j}, d_{j}\right) = W_{j} f_{(j-1)}$, and $f_{(j-1)} = s_{(j-1)}$, for $j = 1, 2, \dots, L$. The transform operator is defined as,
\begin{equation}
    W_{[L]} = W_{L} W_{(L-1)} \dots W_{1},
\end{equation}
where, $W_{j} = 
    \begin{bmatrix}
        I & U_{j} \\ 
        0 & I
    \end{bmatrix}
    \begin{bmatrix}
        I & 0 \\ 
        -P_{j} & I
    \end{bmatrix},$
and $P_{j}$ and $U_{j}$ represent the predict and update operators at level $j$, respectively. At any level $j$, the wavelet transform is given by,
\begin{equation*}
    \begin{bmatrix}
        s_{j} \\ d_{j}
    \end{bmatrix}
    = W_{j} f_{(j-1)} = 
    \begin{bmatrix}
        I & U_{j} \\ 
        0 & I
    \end{bmatrix}
    \begin{bmatrix}
        I & 0 \\ 
        -P_{j} & I
    \end{bmatrix}
    \begin{bmatrix}
        \mathrm{even}_{(j-1)} \\ 
        \mathrm{odd}_{(j-1)}
    \end{bmatrix}.
\end{equation*}
Similarly, the second-generation inverse wavelet transform is given by,
    \begin{equation}\label{eq:2ndgen_wavelet_inverse}
    \begin{split}
        f = &W_{[L]}^{-1}\left(s_{L}, d_{L}, d_{(L-1)}, \dots, d_{1}\right),\\
        &\text{where, } W_{[L]}^{-1} = W_{1}^{-1}\dots W_{(L-1)}^{-1} W_{L}^{-1}.
    \end{split}
    \end{equation}
    At any level $j$, it is given by,
    $$f_{(j-1)} = \left(W_{j}\right)^{-1}\begin{bmatrix}
            s_{j}\\d_{j}
        \end{bmatrix} = \begin{bmatrix}
            I & 0 \\ P_{j} & I
        \end{bmatrix}\begin{bmatrix}
            I & -U_{j} \\ 0 & I
        \end{bmatrix}\begin{bmatrix}
            s_{j}\\ d_{j}            
        \end{bmatrix}.$$

\begin{remark}
It is important to note that second-generation wavelet transforms are generally non-orthogonal.
\end{remark}

\begin{example}
    To explain 
    second-generation wavelets, we consider the operators $U = 0$ and $P$ as the interpolating operator. Since $U=0$, the even signals $\{0, 2, 4, 6\}$ (illustrated in {blue} color in Fig.~\ref{fig:2g-wavelet-working}) are retained as the smooth coefficients $\mathcal{A}_{s_1}$ at level $1$, while the detail coefficients, $\mathcal{A}_{d_1}$ (shown in {red} color in Fig.~\ref{fig:2g-wavelet-working}), are computed as the difference between the odd signals $\{1, 3, 5, 7\}$ and their neighboring interpolates. Proceeding to three levels of wavelet decomposition (because we initially have $2^3=8$ signals), the decomposition results in a single smooth coefficient, $\mathcal{A}_{s_3}$, and a single detail coefficient, $\mathcal{A}_{d_3}$. 
    
    For reconstruction, starting from the smoothest coefficient, $\mathcal{A}_{s_3}$, and all the detail coefficients $\{\mathcal{A}_{d_1}, \mathcal{A}_{d_2}, \mathcal{A}_{d_3}\}$, the process involves interpolating the smooth coefficient, $\mathcal{A}_{s_3}$, and then adding the detail coefficient from the last level, $\mathcal{A}_{d_3}$, to obtain the next smooth coefficient, $\mathcal{A}_{s_2}$. Continuing this process step-by-step, we eventually recover the original signal, $\mathcal{A}$. This example is visually illustrated in Fig.~\ref{fig:2g-wavelet-working}.
\end{example}
\begin{figure}
    \centering
    \includegraphics[width=\linewidth]{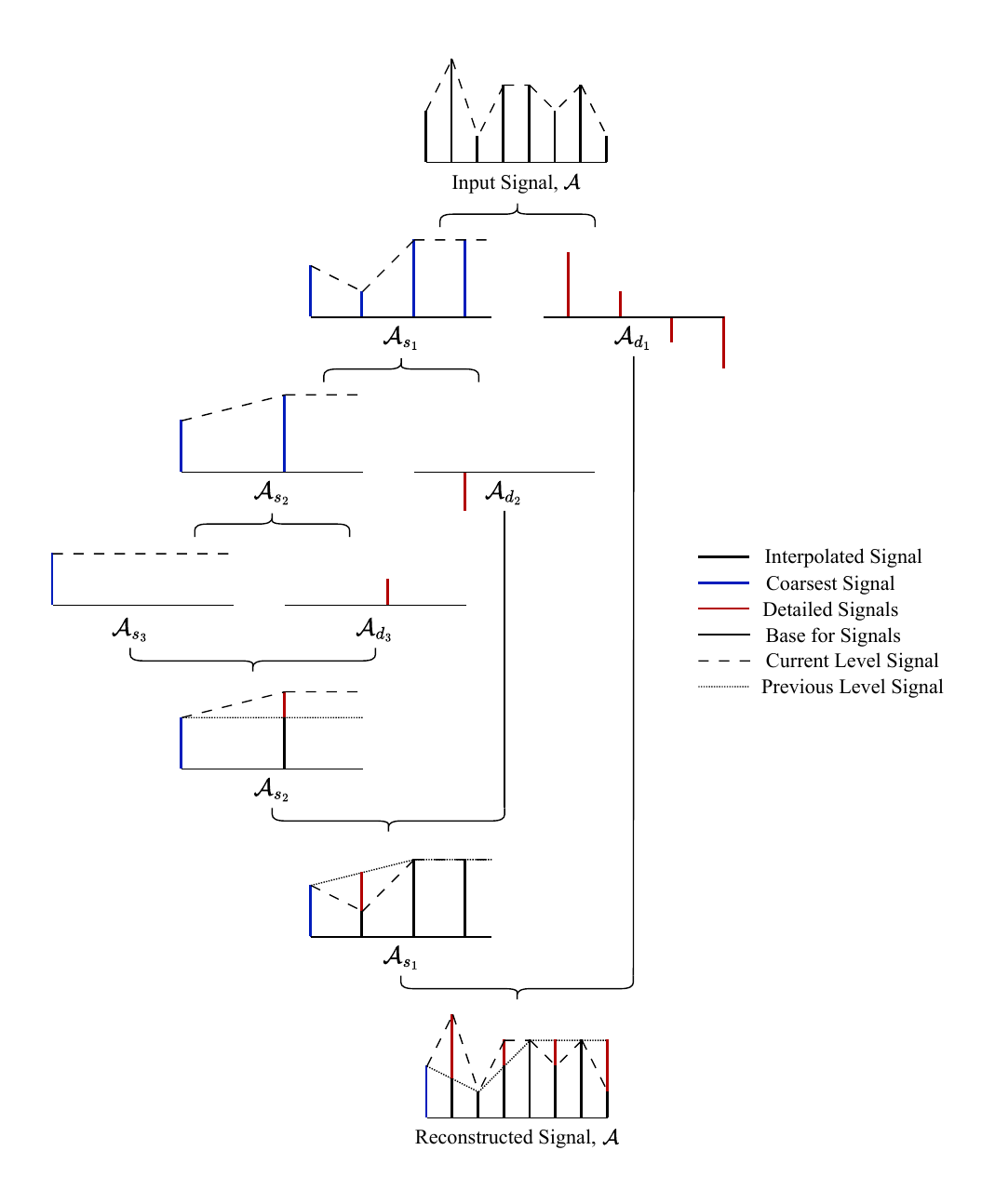}
    \caption{Visual representation of a simple second-generation wavelet decomposition and reconstruction.}
    \label{fig:2g-wavelet-working}
\end{figure}

We now consider the simplest form of the second-generation wavelet transform for the wavelet product, known as the lazy wavelet transform, where $P_{j} = I$ and $U_{j} = \frac{1}{2} I$ for every $j = 1, 2, \dots, L$.

\begin{lemma}[Perfect Reconstruction~\cite{Swe98Thelift}]\label{lem:perfect_reconstruction}
    For all $f \in \mathbb{R}^{p}$ with $p = 2^L m$, where $m, L \in \mathbb{N}$, the second-generation wavelet transform $W_{[L]}$ satisfies the perfect reconstruction property, 
    \begin{equation}
        W_{[L]}^{-1}(W_{[L]}(f)) = f.            
    \end{equation}
\end{lemma}

\begin{lemma}[Linearity]\label{lem:linearity}
    The wavelet transform defined by $W_{[L]}$ for lazy wavelets is linear, that is, for any $f,g \in \mathbb{R}^p$ with $p = 2^L m$, where $m, L \in \mathbb{N}$ and scalars $\alpha, \beta \in \mathbb{R}$, the following holds,
    \begin{equation}
        W_{[L]}(\alpha f + \beta g) = \alpha W_{[L]}(f) + \beta W_{[L]}(g).
    \end{equation}
\end{lemma}


\subsection{Second-generation wavelet-inspired tensor product ($w$-product)}
Let $\mathcal{A}\in\mathbb{R}^{n_1\times n_2\times p}$ and $\mathcal{B}\in\mathbb{R}^{n_2\times n_3\times p}$ be two tensors, where $p = 2^Lm$ for $m,L\in\mathbb{N}$. We define the wavelet transformation operator, $W_{[L]}:\mathbb{R}^p\to\mathbb{R}^{\frac{p}{2}}\times\mathbb{R}^{\frac{p}{4}}\times \cdots \times\mathbb{R}^{\frac{p}{2^L}}$.  We now define a new type of multiplication between tensors,  termed the $w$-product as follows. 
\begin{definition}[$w$-product]\label{def:wproduct}
    Let $\mathcal{A} \in \mathbb{R}^{n_1 \times n_2 \times p}$ and $\mathcal{B} \in \mathbb{R}^{n_2 \times n_3 \times p}$  with $p = 2^L m$, where $m, L \in \mathbb{N}$, then the $w$-product $\mathcal{A}\star_w\mathcal{B} \in \mathbb{R}^{n_1 \times n_3 \times p}$ is defined as,
    \begin{equation}
        \mathcal{A}\star_w\mathcal{B} = W_{[L]}^{-1}\left(W_{[L]}(\mathcal{A})\Delta W_{[L]}(\mathcal{B})\right),
    \end{equation}
    where, $W_{[L]}$ and $W_{[L]}^{-1}$ are given by Eq. \eqref{eq:2ndgen_wavelet_forward} and Eq. \eqref{eq:2ndgen_wavelet_inverse}, and $\Delta$ represents the face-product.
\end{definition}
Computation of the $w$-product, $\mathcal{A}\star_w\mathcal{B}$ follows the systematic procedure discussed in Algorithm~\ref{alg:w_product}. It is worth mentioning that the $w$-product is associative due to the linearity of the wavelet transform and matrix face-products in the wavelet-domain, i.e., $(\mathcal{A} \star_w \mathcal{B}) \star_w \mathcal{C} = \mathcal{A} \star_w (\mathcal{B} \star_w \mathcal{C})$, for suitable dimensions of $\mathcal{A}$, $\mathcal{B}$, and $\mathcal{C}$. Proof is given in Theorem~\ref{thrm:associativity}.

The $w$-product algorithm provides a hierarchical framework for multiplying third-order tensors through a multilevel wavelet transformation. Given tensors $\mathcal{A} \in \mathbb{R}^{n_1 \times n_2 \times p}$ and $\mathcal{B} \in \mathbb{R}^{n_2 \times n_3 \times p}$, where $p=2^L m$, the method successively applies $L$ levels of discrete wavelet decomposition along the third mode to extract the detail and smooth components at each scale. In the forward pass, the algorithm computes forward lazy wavelet transforms along the third dimension. Subsequently, at each level $j$, the face product ($\mathcal{A}\Delta\mathcal{B}$) was computed between same coefficients of $\mathcal{A}$ and $\mathcal{B}$. Subsequently, inverse wavelet reconstruction is employed in a bottom-up manner to synthesize the resulting tensor $\mathcal{C} = \mathcal{A} \star_w \mathcal{B}$, effectively integrating information across all scales. For ease of understanding, the working of the $w$-product is shown in Fig.~\ref{fig:wproduct} using the third dimension $p=2^3$.
\begin{figure*}
	\centering
	\includegraphics[width=0.9\textwidth]{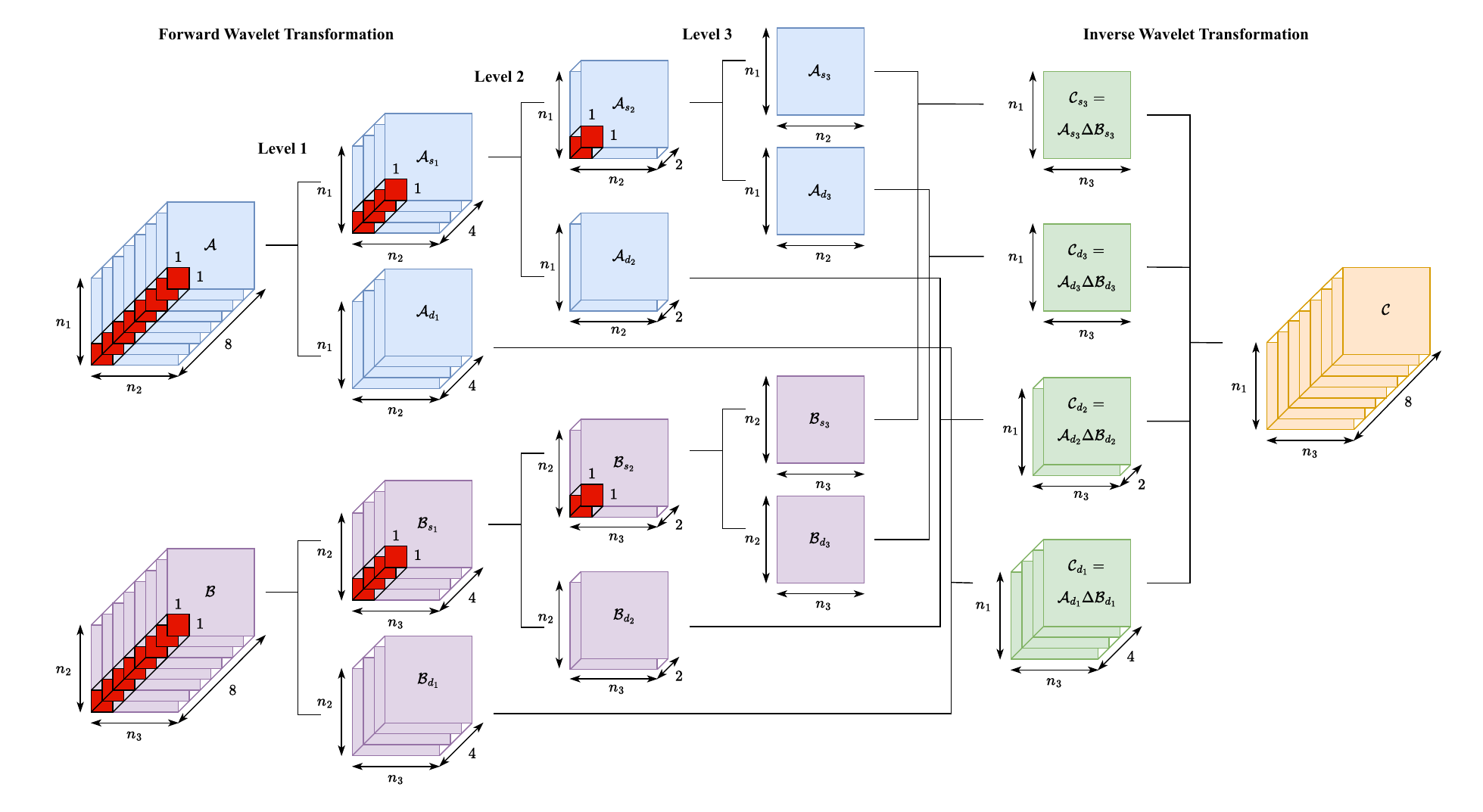}
	\caption{Flowchart explaining the working of the proposed wavelet product for tensors with assumption $p=2^3$.}
	\label{fig:wproduct}
\end{figure*}

\begin{algorithm}
\small
\caption{$w$-product for third-Order tensors}
\label{alg:w_product}
\begin{algorithmic}[1]
\REQUIRE Tensor $\mathcal{A} \in \mathbb{R}^{n_1 \times n_2 \times p}$, $\mathcal{B} \in \mathbb{R}^{n_2 \times n_3 \times p}$ with $p=2^L m$ where $m,L \in \mathbb{N}$

\STATE Initialize: $\mathcal{A}_{s_{0}} = \mathcal{A}$, $\mathcal{B}_{s_{0}} = \mathcal{B}$
\FOR{$j=1$ {\textbf{to}} $L$}
    \STATE Set: $\ell = p/\left(2^{j-1}\right)$

    \COMMENT{{\color{CornflowerBlue}{$\triangleright$ Forward Wavelet Transform for $\mathcal{A}$}}}
    \STATE Detail coefficients: $\mathcal{A}_{d_{j}} = \mathcal{A}_{s_{j-1}}[:,:,1:2:(\ell-1)] - \mathcal{A}_{s_{j-1}}[:,:,2:2:\ell]$
    \STATE Smooth coefficients: $\mathcal{A}_{s_{j}} = \mathcal{A}_{s_{j-1}}[:,:,2:2:\ell] + \mathcal{A}_{d_{j}} / 2$

    \COMMENT{{\color{Purple}{$\triangleright$ Forward Wavelet Transform for $\mathcal{B}$}}}
    \STATE Detail coefficients: $\mathcal{B}_{d_{j}} = \mathcal{B}_{s_{j-1}}[:,:,1:2:(\ell-1)] - \mathcal{B}_{s_{j-1}}[:,:,2:2:\ell]$
    \STATE Smooth coefficients: $\mathcal{B}_{s_{j}} = \mathcal{B}_{s_{j-1}}[:,:,2:2:\ell] + \mathcal{B}_{d_{j}} / 2$
\ENDFOR

\COMMENT{{\color{LimeGreen}{$\triangleright$ Tensor Product in Wavelet Domain}}}
\FOR{$j=1$ {\textbf{to}} $L$}
    \FOR{$k=1$ {\textbf{to}} $p/2^{j}$}
        \STATE $\mathcal{C}_{d_j}[:,:,k] \gets \mathcal{A}_{d_j}[:,:,k] \mathcal{B}_{d_j}[:,:,k]$
        \IF{$j$ {\textbf{is}} $L$}
            \STATE $\mathcal{C}_{s_L}[:,:,k] \gets \mathcal{A}_{s_L}[:,:,k] \mathcal{B}_{s_L}[:,:,k]$
        \ENDIF
    \ENDFOR
\ENDFOR

\COMMENT{{\color{YellowOrange}{$\triangleright$ Inverse Wavelet Transform}}}
\STATE Initialize: $\mathcal{C} \in \mathbb{R}^{n_1\times n_3 \times p}$

\FOR{$j = L$ {\textbf{downto}} $1$}
    \STATE Set: $\ell = p/\left(2^{j-1}\right)$
    \STATE $\mathcal{C}_{s_{j-1}}[:,:,2:2:\ell] = \mathcal{C}_{s_j} - \mathcal{C}_{d_j} / 2$
    \STATE $\mathcal{C}_{s_{j-1}}[:,:,1:2:(\ell-1)] = \mathcal{C}_{d_j} + \mathcal{C}_{s_{j-1}}[:,:,2:2:\ell]$
\ENDFOR
\STATE Finalize: $\mathcal{C} = \mathcal{C}_{s_0}$

\RETURN $\mathcal{C}$
\end{algorithmic}
\end{algorithm}

\begin{example}
    Let 
    $\mathcal{A}\in\mathbb{R}^{2\times3\times4}$ and $\mathcal{B}\in\mathbb{R}^{3\times2\times4}$ with frontal slices
{\small
$$\mathcal{A}^{(1)} = \begin{pmatrix}
3 & 0 & 2 \\
3 & 3 & 0
\end{pmatrix},~~ \mathcal{A}^{(2)} = \begin{pmatrix}
2 & 0 & 0 \\
1 & 1 & 1
\end{pmatrix},$$ 
$$\mathcal{A}^{(3)} = \begin{pmatrix}
0 & 1 & 4 \\
3 & 3 & 2
\end{pmatrix},~~ \mathcal{A}^{(4)} = \begin{pmatrix}
3 & 0 & 5 \\
2 & 2 & 5
\end{pmatrix}.$$
$$\mathcal{B}^{(1)} = \begin{pmatrix}
1 & 3 \\
1 & 2 \\
3 & 1
\end{pmatrix},~~ \mathcal{B}^{(2)} = \begin{pmatrix}
5 & 0 \\
0 & 0 \\
0 & 3
\end{pmatrix},$$
$$\mathcal{B}^{(3)} = \begin{pmatrix}
3 & 4 \\
0 & 4 \\
5 & 0
\end{pmatrix},~~ \mathcal{B}^{(4)} = \begin{pmatrix}
5 & 5 \\
5 & 4 \\
5 & 3
\end{pmatrix}.$$}

\noindent
Computing the level $1$ forward wavelet transform for tensor $\mathcal{A}$ and $\mathcal{B}$, we get,
{\small
$$\mathcal{A}_{s_1}^{(1)} = \begin{pmatrix}
2.5 & 0.0 & 1.0 \\
2.0 & 2.0 & 0.5
\end{pmatrix},~~\mathcal{A}_{s_1}^{(2)} = \begin{pmatrix}
1.5 & 0.5 & 4.5 \\
2.5 & 2.5 & 3.5
\end{pmatrix},$$

$$\mathcal{A}_{d_1}^{(1)} = \begin{pmatrix}
1 & 0 & 2 \\
2 & 2 & -1
\end{pmatrix},~~ \mathcal{A}_{d_1}^{(2)} = \begin{pmatrix}
-3 & 1 & -1 \\
1 & 1 & -3
\end{pmatrix}.$$

$$\mathcal{B}_{s_1}^{(1)} = \begin{pmatrix}
3.0 & 1.5 \\
0.5 & 1.0 \\
1.5 & 2.0
\end{pmatrix},~~ \mathcal{B}_{s_1}^{(2)} = \begin{pmatrix}
4.0 & 4.5 \\
2.5 & 4.0 \\
5.0 & 1.5
\end{pmatrix},$$

$$\mathcal{B}_{d_1}^{(1)} = \begin{pmatrix}
-4 & 3 \\
1 & 2 \\
3 & -2
\end{pmatrix},~~ \mathcal{B}_{d_1}^{(2)} = \begin{pmatrix}
-2 & -1 \\
-5 & 0 \\
0 & -3
\end{pmatrix}.$$}

\noindent
Computing the level $2$ forward wavelet transform for tensor $\mathcal{A}_{s_1}$ and $\mathcal{B}_{s_1}$, we get,
{\small
$$\mathcal{A}_{s_2} = \begin{pmatrix}
2.00 & 0.25 & 2.75 \\
2.25 & 2.25 & 2.00
\end{pmatrix},~~\mathcal{A}_{d_2} = \begin{pmatrix}
1.0 & -0.5 & -3.5 \\
-0.5 & -0.5 & -3.0
\end{pmatrix}.$$
$$\mathcal{B}_{s_2} = \begin{pmatrix}
3.50 & 3.00 \\
1.50 & 2.50 \\
3.25 & 1.75
\end{pmatrix}, ~~\mathcal{B}_{d_2} = \begin{pmatrix}
-1.0 & -3.0 \\
-2.0 & -3.0 \\
-3.5 & 0.5
\end{pmatrix}.$$}

\noindent
Now, by multiplying $\mathcal{A}_{s_2}$ with $\mathcal{B}_{s_2}$ face-wise and similarly multiplying $\mathcal{A}_{d_2}$ with $\mathcal{B}_{d_2}$ and $\mathcal{A}_{d_1}$ with $\mathcal{B}_{d_1}$, we get,
{\small
$$\mathcal{C}_{d_1}^{(1)} = \begin{pmatrix}
2 & -1 \\
-9 & 12
\end{pmatrix},~~ \mathcal{C}_{d_1}^{(2)} = \begin{pmatrix}
1 & 6 \\
-7 & 8
\end{pmatrix}.$$
$$\mathcal{C}_{s_2} = \begin{pmatrix}
16.3125 & 11.4375 \\
17.7500 & 15.8750
\end{pmatrix},~~\mathcal{C}_{d_2} = \begin{pmatrix}
12.2500 & -3.2500 \\
12.0000 & 1.5000
\end{pmatrix}.$$}

\noindent
Now, by applying inverse wavelet transformation to the tensors $\mathcal{C}_{s_2}$ and $\mathcal{C}_{d_2}$, we get the resulting tensor $\mathcal{C}_{s_1}$ as,
{\small
$$\mathcal{C}_{s_1}^{(1)} = \begin{pmatrix}
22.4375 & 9.8125 \\
23.7500 & 16.6250
\end{pmatrix},~~\mathcal{C}_{s_1}^{(2)} = \begin{pmatrix}
10.1875 & 13.0625 \\
11.7500 & 15.1250
\end{pmatrix}.$$}

\noindent
Furthermore, by applying inverse wavelet transformation to tensors $\mathcal{C}_{s_1}$ and $\mathcal{C}_{d_1}$, we get the resulting tensor $\mathcal{C}$,
{\small
$$\mathcal{C}^{(1)} = \begin{pmatrix}
23.4375 & 9.3125 \\
19.2500 & 22.6250
\end{pmatrix},~~\mathcal{C}^{(2)} = \begin{pmatrix}
21.4375 & 10.3125 \\
28.2500 & 10.6250
\end{pmatrix},$$
$$\mathcal{C}^{(3)} = \begin{pmatrix}
10.6875 & 16.0625 \\
8.2500 & 19.1250
\end{pmatrix},~~\mathcal{C}^{(4)} = \begin{pmatrix}
9.6875 & 10.0625 \\
15.2500 & 11.1250
\end{pmatrix}.$$}
\end{example}

The linearity of $W_{[L]}$ along the third dimension is preserved according to Lemma~\ref{lem:linearity}, since the wavelet transformation 
is applied independently along the third dimension. Some of the important properties of the $w$-product are as follows,

\begin{definition}[Transpose]\label{def:transpose}
    The transpose of the tensor $\mathcal{A}\in \mathbb{R}^{n_1\times n_2\times p}$ is
    $\mathcal{A}^\top[i,j,:] = \mathcal{A}[j,i,:].$
\end{definition}

Due to the application of forward and inverse wavelet transformations along the third dimension, the transpose operation remains invariant to wavelet transformation, which does not hold in the $t$-product framework. 

\begin{definition}[Identity]\label{def:identity}
    The identity tensor $\mathcal{I} \in \mathbb{R}^{n \times n \times p}$ with $p = 2^L m$, where $m, L \in \mathbb{N}$, in the wavelet domain, is,
    \begin{equation}
    \begin{split}
        &\mathcal{I}_{s_j}[:, :, k] = I_n, \\ 
        &\mathcal{I}_{d_j}[:, :, k] = I_n, \quad \text{for all } j,k.
    \end{split}
    \end{equation}
    Then, we have $\mathcal{A} \star_w \mathcal{I} = \mathcal{I} \star_w \mathcal{A} = \mathcal{A}$. The tensor $\mathcal{I}$ can be computed easily using the invertibility property of second-generation wavelets, as given by Eq.~\eqref{eq:2ndgen_wavelet_inverse}.
\end{definition}

\begin{definition}[Inverse]\label{def:inverse}
    The inverse tensor $\mathcal{A}^{-1} \in \mathbb{R}^{n \times n \times p}$ with $p = 2^L m$, where $m, L \in \mathbb{N}$, in the wavelet domain is,
    \begin{equation}
    \begin{split}
        &\mathcal{A}_{s_j}^{-1}[:, :, k] = \left(\mathcal{A}_{s_j}[:, :, k]\right)^{-1}, \\ 
        &\mathcal{A}_{d_j}^{-1}[:, :, k] = \left(\mathcal{A}_{d_j}[:, :, k]\right)^{-1}, \quad \text{for all } j,k.
    \end{split}
    \end{equation}
    Then, it holds that $\mathcal{A} \star_w \mathcal{A}^{-1} = \mathcal{A}^{-1} \star_w \mathcal{A} = \mathcal{I}$. Similar to the identity tensor, the tensor $\mathcal{A}^{-1}$ can be computed using the invertibility property of second-generation wavelets as given by Eq.~\ref{eq:2ndgen_wavelet_inverse}. The double inverse property directly holds from here, $(\mathcal{A}^{-1})^{-1} = \mathcal{A}$. Proof given in the following Lemma.
\end{definition}
\begin{lemma}[Inverse Property]\label{lem:inverse}
    For a tensor $\mathcal{A} \in \mathbb{R}^{n \times n \times p}$ with $p = 2^L m$, where $m,L \in \mathbb{N}$, it holds that
    \begin{equation}
        \left(\mathcal{A}^{-1}\right)^{-1} = \mathcal{A}.
    \end{equation}
\end{lemma}
\begin{proof}
    From Definition~\ref{def:inverse}, at any level $j=1,2,\dots,L$, the inverse components satisfy,
    $$(\mathcal{A}^{-1})_{d_j}[:,:,k] = (\mathcal{A}_{d_j}[:,:,k])^{-1}.$$
    Taking the inverse again, we get,
    \begin{equation*}
    \begin{split}
        ((\mathcal{A}^{-1})^{-1})_{d_j}[:,:,k] &= \bigl((\mathcal{A}^{-1})_{d_j}[:,:,k]\bigr)^{-1} \\
        &= \bigl((\mathcal{A}_{d_j}[:,:,k])^{-1}\bigr)^{-1} = \mathcal{A}_{d_j}[:,:,k].
    \end{split}
    \end{equation*}
    Similarly,
    $$((\mathcal{A}^{-1})^{-1})_{s_L}[:,:,k] = \mathcal{A}_{s_L}[:,:,k].$$
    Applying the inverse wavelet transform then yields,
    $$(\mathcal{A}^{-1})^{-1} = W_{[L]}^{-1}(\mathcal{A}_{s_L}, \{\mathcal{A}_{d_j}\}_{j=1}^L) = \mathcal{A}.$$
    Therefore, the lemma is proved.
\end{proof}

\begin{definition}[Orthogonal Tensor]\label{def:othotensor}
    The orthogonal tensor $\mathcal{Q} \in \mathbb{R}^{n \times n \times p}$ with $p = 2^L m$, where $m, L \in \mathbb{N}$, in the wavelet domain, is defined using orthogonal matrices $Q_n$,
    \begin{equation}
    \begin{split}
        &\mathcal{Q}_{s_j}[:, :, k] = Q_n, \\ 
        &\mathcal{Q}_{d_j}[:, :, k] = Q_n, \quad \text{for all } j,k.
    \end{split}
    \end{equation}
    such that $Q_n^\top Q_n = Q_nQ_n^\top = I_n$ where, $I_n,Q_n\in\mathbb{R}^{n\times n}$ and $I_n$ is identity matrix. The orthogonal tensor $\mathcal{Q}$ can be computed using the invertibility property given by Eq.~\eqref{eq:2ndgen_wavelet_inverse}.
\end{definition}

\begin{theorem}[Transpose Invariance Property]\label{thrm:transposeProp}
    For a tensor $\mathcal{A} \in \mathbb{R}^{n_1 \times n_2 \times p}$, the transpose operation commutes with the wavelet transform and its inverse in the spatial domain, i.e.,
    \begin{equation}
    \begin{split}
        &W_{[L]}(\mathcal{A}^\top) = \bigl(W_{[L]}(\mathcal{A})\bigr)^\top, \\
        &W_{[L]}^{-1}(\mathcal{A}^\top) = \bigl(W_{[L]}^{-1}(\mathcal{A})\bigr)^\top.
    \end{split}
    \end{equation}
\end{theorem}
\begin{proof}
    It is known that the operators $W_j$ and $W_j^{-1}$ act independently along the third dimension, while the transpose swaps the first two dimensions. Hence, for level $1$, we have,
    \begin{equation*}
    \begin{split}
        (\mathcal{A}^\top)&_{s_1}[i,j,k] = \mathcal{A}^\top[i,j,2k-1] - \mathcal{A}^\top[i,j,2k] \\
        &= \mathcal{A}[j,i,2k-1] - \mathcal{A}[j,i,2k] = (\mathcal{A}_{s_1})^\top[i,j,k].
    \end{split}
    \end{equation*}
    Similarly, $(\mathcal{A}^\top)_{d_1}[i,j,k] = (\mathcal{A}_{d_1})^\top[i,j,k]$. By induction, it follows that for all levels $j$,
    $$(\mathcal{A}^\top)_{s_j} = (\mathcal{A}_{s_j})^\top \quad \text{and} \quad (\mathcal{A}^\top)_{d_j} = (\mathcal{A}_{d_j})^\top.$$
    Combining all levels, we obtain,
    $W_{[L]}(\mathcal{A}^\top) = (W_{[L]}(\mathcal{A}))^\top.$
    Since the inverse wavelet transform is also linear and operates independently on the third dimension, we have,
    $$W_{[L]}^{-1}(\mathcal{A}^\top) = (W_{[L]}^{-1}(\mathcal{A}))^\top.$$
    This completes the proof.
\end{proof}

\begin{theorem}[Transpose Reversal Property]\label{thrm:transpose_reversal}
    For two tensors $\mathcal{A} \in \mathbb{R}^{n_1 \times n_2 \times p}$ and $\mathcal{B} \in \mathbb{R}^{n_2 \times n_3 \times p}$, the transpose of their wavelet product satisfies,
    \begin{equation}
        \left(\mathcal{A} \star_w \mathcal{B}\right)^\top = \mathcal{B}^\top \star_w \mathcal{A}^\top.
    \end{equation}
\end{theorem}
\begin{proof}
    Assume the wavelet transforms,
    $$W_{[L]}(\mathcal{A}) = (\mathcal{A}_{s_L}, \{\mathcal{A}_{d_j}\}_{j=1}^L),~~W_{[L]}(\mathcal{B}) = (\mathcal{B}_{s_L}, \{\mathcal{B}_{d_j}\}_{j=1}^L).$$
    By definition of the product $\mathcal{C} = \mathcal{A} \star_w \mathcal{B}$, for each level $j = 1, 2, \dots, L$,
    $$\mathcal{C}_{d_j} = \mathcal{A}_{d_j} \Delta \mathcal{B}_{d_j}, \quad \mathcal{C}_{s_L} = \mathcal{A}_{s_L} \Delta \mathcal{B}_{s_L},$$
    and,
    $$\mathcal{C} = W_{[L]}^{-1} \left( \mathcal{C}_{s_L}, \{\mathcal{C}_{d_j}\}_{j=1}^L \right).$$
    From Theorem~\ref{thrm:transposeProp}, we have
    $\mathcal{C}^\top = W_{[L]}^{-1} \left( \mathcal{C}_{s_L}^\top, \{\mathcal{C}_{d_j}^\top\}_{j=1}^L \right).$
    For each level $j$ and slice $k$, using transpose properties of matrix multiplication,
    \begin{equation*}
        \begin{split}
            (\mathcal{C}_{d_j}[:,:,k])^\top &= (\mathcal{A}_{d_j}[:,:,k]\mathcal{B}_{d_j}[:,:,k])^\top\\ &= (\mathcal{B}_{d_j}[:,:,k])^\top(\mathcal{A}_{d_j}[:,:,k])^\top\\ &= (\mathcal{B}^\top)_{d_j}[:,:,k](\mathcal{A}^\top)_{d_j}[:,:,k]
        \end{split}
    \end{equation*}
    Similarly, $(\mathcal{C}_{s_L}[:,:,k])^\top = (\mathcal{B}^\top)_{s_L}[:,:,k] (\mathcal{A}^\top)_{s_L}[:,:,k].$
    Therefore,
    $$\mathcal{C}^\top = W_{[L]}^{-1} \left( W_{[L]}(\mathcal{B}^\top \star_w \mathcal{A}^\top) \right) = \mathcal{B}^\top \star_w \mathcal{A}^\top,$$
    which completes the proof.
\end{proof}

\begin{theorem}[Inverse Reversal Property]\label{thrm:inverse_reversal}
    For two invertible tensors $\mathcal{A} \in \mathbb{R}^{n \times n \times p}$ and $\mathcal{B} \in \mathbb{R}^{n \times n \times p}$, the inverse of their wavelet product $\star_w$ satisfies
    \begin{equation}
        \left(\mathcal{A} \star_w \mathcal{B}\right)^{-1} = \mathcal{B}^{-1} \star_w \mathcal{A}^{-1}.
    \end{equation}
\end{theorem}
\begin{proof}
    The proof follows by applying a similar argument as in Theorem~\ref{thrm:transpose_reversal} together with the definition of the inverse as given in Definition~\ref{def:inverse}. 
    Specifically, the reversal of the order in the wavelet product inverse is a direct consequence of these properties.
\end{proof}

\begin{theorem}[Associativity]\label{thrm:associativity}
    For three tensors $\mathcal{A} \in \mathbb{R}^{n_1 \times n_2 \times p}$, $\mathcal{B} \in \mathbb{R}^{n_2 \times n_3 \times p}$, and $\mathcal{C} \in \mathbb{R}^{n_3 \times n_4 \times p}$, the wavelet product satisfies associativity,
    \begin{equation}
        \left(\mathcal{A} \star_w \mathcal{B}\right) \star_w \mathcal{C} = \mathcal{A} \star_w \left(\mathcal{B} \star_w \mathcal{C}\right).
    \end{equation}
\end{theorem}
\begin{proof}
    At any level $j=1,2,\dots,L$, consider the left-hand side (LHS),
    \begin{equation*}
    \begin{split}
        ((\mathcal{A} \star_w \mathcal{B}) \star_w \mathcal{C})_{d_j}&[:,:,k] \\&= (\mathcal{A} \star_w \mathcal{B})_{d_j}[:,:,k] \, \mathcal{C}_{d_j}[:,:,k] \\
        &= (\mathcal{A}_{d_j}[:,:,k] \, \mathcal{B}_{d_j}[:,:,k]) \, \mathcal{C}_{d_j}[:,:,k].
    \end{split}
    \end{equation*}
    Now consider the right-hand side (RHS),
    \begin{equation*}
    \begin{split}
        (\mathcal{A} \star_w (\mathcal{B} \star_w \mathcal{C}))_{d_j}&[:,:,k]\\ &= \mathcal{A}_{d_j}[:,:,k] \, (\mathcal{B} \star_w \mathcal{C})_{d_j}[:,:,k] \\
        &= \mathcal{A}_{d_j}[:,:,k] \, (\mathcal{B}_{d_j}[:,:,k] \, \mathcal{C}_{d_j}[:,:,k]).
    \end{split}
    \end{equation*}
    By the associativity of matrix multiplication, it follows that,
    $$(\mathcal{A}_{d_j} \Delta \mathcal{B}_{d_j}) \Delta \mathcal{C}_{d_j} = \mathcal{A}_{d_j} \Delta (\mathcal{B}_{d_j} \Delta \mathcal{C}_{d_j}),
    \quad \text{for all levels } j.$$
    Hence,
    $$((\mathcal{A} \star_w \mathcal{B}) \star_w \mathcal{C})_{d_j} = (\mathcal{A} \star_w (\mathcal{B} \star_w \mathcal{C}))_{d_j}.$$
    Similarly, for the smooth coefficient at level $L$,
    $$((\mathcal{A} \star_w \mathcal{B}) \star_w \mathcal{C})_{s_L} = (\mathcal{A} \star_w (\mathcal{B} \star_w \mathcal{C}))_{s_L}.$$
    Applying the inverse wavelet transform on all levels yields,
    $$\left(\mathcal{A} \star_w \mathcal{B}\right) \star_w \mathcal{C} = \mathcal{A} \star_w \left(\mathcal{B} \star_w \mathcal{C}\right).$$
    This completes the proof.
\end{proof}


\begin{theorem}[Singular Value Decomposition]\label{thrm:svd}
    Let $\mathcal{A}\in\mathbb{R}^{n_1\times n_2\times p}$ be a third-order tensor with $p=2^Lm$, where, $m,L\in\mathbb{N}$. Then, there exist tensors $\mathcal{U}\in\mathbb{R}^{n_1\times r\times p}$, $\Sigma\in\mathbb{R}^{r\times r\times p}$, and $\mathcal{V}\in\mathbb{R}^{n_2\times r\times p}$ such that $\mathcal{A}$ admits a low-rank approximation,
    \begin{equation}
        \mathcal{A}^{(r)} \approx \mathcal{U}\star_w \Sigma\star_w \mathcal{V}^\top
    \end{equation}
    where $\star_w$ is the $w$-product and, $\mathcal{U}$ and $\mathcal{V}$ are orthogonal tensors. For any rank $r\leq\min\{n_1,n_2\}$, $\mathcal{A}^{(r)}$ satisfies,
    {\small\begin{equation}
        \|\mathcal{A} - \mathcal{A}^{(r)}\|_F\leq \sqrt{\sum_{i = r+1}^{\min\{n_1,n_2\}} \left(\sum_{j,k}\left(\sigma_{{d_j},k,i}^2\right) + \sum_k \left(\sigma_{{s_L},k,i}^2\right)\right) }
    \end{equation}}
\end{theorem}
\begin{proof}
    A third-order tensor $\mathcal{A} \in \mathbb{R}^{n_1 \times n_2 \times p}$ with $p = 2^L m$, where $m, L \in \mathbb{N}$, admits a wavelet decomposition as follows,
    $$\mathcal{A} = W_{[L]}^{-1}(\mathcal{A}_{s_L}, \{\mathcal{A}_{d_j}\}_{j=1}^L)$$
    where $\mathcal{A}_{s_L} \in \mathbb{R}^{n_1 \times n_2 \times m}$ and $\mathcal{A}_{d_j} \in \mathbb{R}^{n_1 \times n_2 \times \frac{p}{2^j}}$ for $j=1,2,\dots,L$. Now each of the face matrices follows a low rank SVD decomposition as follows,
    \begin{equation}\label{eq:svd_proof1}
        \mathcal{A}_{s_L}[:,:,k] = \mathcal{U}_{s_L}[:,:,k] \Sigma_{s_L}[:,:,k] \mathcal{V}^\top_{s_L}[:,:,k],
    \end{equation}
    for $k=1,2,\dots,m$ where $\mathcal{U}_{s_L}[:,:,k] \in \mathbb{R}^{n_1 \times r}, \Sigma_{s_L}[:,:,k] \in \mathbb{R}^{r \times r}$ and $\mathcal{V}_{s_L}[:,:,k] \in \mathbb{R}^{n_2 \times r}$ for all $k$. Also,
    \begin{equation}\label{eq:svd_proof2}
        \mathcal{A}_{d_j}[:,:,k] = \mathcal{U}_{d_j}[:,:,k] \Sigma_{d_j}[:,:,k] \mathcal{V}^\top_{d_j}[:,:,k],
    \end{equation}
    for $k=1,2,\dots, \frac{p}{2^j}$ and $j=1,2,\dots,L$ where $\mathcal{U}_{d_j}[:,:,k] \in \mathbb{R}^{n_1 \times r}, \Sigma_{d_j}[:,:,k] \in \mathbb{R}^{r \times r}$ and $\mathcal{V}_{d_j}[:,:,k] \in \mathbb{R}^{n_2 \times r}$ for all $k$. From Eq.~\eqref{eq:svd_proof1} and Eq.~\eqref{eq:svd_proof2}, applying inverse wavelet transformations, we get $\mathcal{U}, \Sigma,$ and $\mathcal{V}$ such that,
    \begin{equation*}    
        \begin{split}
            &\mathcal{U} = W_{[L]}^{-1}(\mathcal{U}_{s_L}, \{\mathcal{U}_{d_j}\}_{j=1}^L),\\
            &\Sigma = W_{[L]}^{-1}(\Sigma_{s_L}, \{\Sigma_{d_j}\}_{j=1}^L),\\
            &\mathcal{V} = W_{[L]}^{-1}(\mathcal{V}_{s_L}, \{\mathcal{V}_{d_j}\}_{j=1}^L),
        \end{split}
    \end{equation*}
    which satisfies the equation
    \begin{equation*}
        \mathcal{A}^{(r)} \approx \mathcal{U} \star_w \Sigma \star_w \mathcal{V}^\top.
    \end{equation*}
    It is important to note that here, $\mathcal{U}$ and $\mathcal{V}$ are orthogonal tensors defined in Definition~\ref{def:othotensor}. This shows that there exists a low-rank SVD decomposition satisfying the given equation.
\end{proof}

\begin{definition}[Moore-Penrose Inverse]\label{defn:pinv}
    Let $\mathcal{A}\in \mathbb{R}^{m\times n\times p}$, then the tensor $\mathcal{X}\in\mathbb{R}^{n\times m\times p}$ with $p = 2^Lm$, where, $m,L\in\mathbb{N}$, satisfying,
    \begin{equation}
        \begin{cases}
            \mathcal{A}\star_w\mathcal{X}\star_w\mathcal{A} = \mathcal{A}\\
            \mathcal{X}\star_w\mathcal{A}\star_w\mathcal{X} = \mathcal{X}\\
            \left(\mathcal{A}\star_w\mathcal{X}\right)^\top = \mathcal{A}\star_w\mathcal{X}\\
            \left(\mathcal{X}\star_w\mathcal{A}\right)^\top = \mathcal{X}\star_w\mathcal{A}
        \end{cases}
    \end{equation}
    is called the Moore-Penrose inverse of the tensor $\mathcal{A}$, and is denoted by $\mathcal{A}^\dagger$.
\end{definition}

\begin{theorem}[Existence and Uniqueness of Moore-Penrose Inverse]
    The Moore-Penrose inverse of the tensor $\mathcal{A}\in\mathbb{R}^{m\times n\times p}$ exists and is unique under the $w$-product.
\end{theorem}
\begin{proof}
    From Definition~\ref{defn:pinv}, Theorem~\ref{thrm:inverse_reversal} and Theorem~\ref{thrm:svd}, it directly follows that
    \begin{equation}
        \mathcal{A}^\dagger = \mathcal{V}\star_w\Sigma^\dagger\star_w\mathcal{U}^\top
    \end{equation}
    where, $\mathcal{U}$ and $\mathcal{V}$ are orthogonal tensors. Hence, the Moore-Penrose inverse of $\mathcal{A}$ exists. 
    
    Now, for uniqueness, assume there exist two distinct Moore-Penrose inverses of $\mathcal{A}$ given by $\mathcal{X},\mathcal{Y}\in\mathbb{R}^{n\times m\times p}$. Subsequently, according to Definition~\ref{defn:pinv} we get,
    \begin{equation*}
        \begin{split}
            \mathcal{X} &= \mathcal{X}\star_w\mathcal{A}\star_w\mathcal{X} = \mathcal{X}\star_w\left(\mathcal{A}\star_w\mathcal{X}\right)^\top = \mathcal{X}\star_w\mathcal{X}^\top\star_w\mathcal{A}^\top\\
            &= \mathcal{X}\star_w\mathcal{X}^\top\star_w\mathcal{A}^\top\star_w\mathcal{Y}^\top\star_w\mathcal{A}^\top= \mathcal{X}\star_w\mathcal{Y}^\top\star_w\mathcal{A}^\top\\
            &= \mathcal{X}\star_w \mathcal{A}\star_w\mathcal{Y} = \mathcal{X}\star_w \mathcal{A}\star_w\mathcal{Y}\star_w \mathcal{A}\star_w\mathcal{Y}\\
            &= \left(\mathcal{X}\star_w\mathcal{A}\right)^\top\star_w\left(\mathcal{Y}\star_w\mathcal{A}\right)^\top\star_w\mathcal{Y}\\
            &= \mathcal{A}^\top\star_w\mathcal{X}^\top\star_w\mathcal{A}^\top\star_w\mathcal{Y}^\top\star_w\mathcal{Y} = \mathcal{A}^\top\star_w\mathcal{Y}^\top\star_w\mathcal{Y}\\
            &= \mathcal{Y}\star_w\mathcal{A}\star_w\mathcal{Y} = \mathcal{Y}.
        \end{split}
    \end{equation*}
    which proves the uniqueness of the Moore-Penrose inverse.
\end{proof}

\begin{remark}\label{rem:neqInverseRev}
   In general, the inverse reversal property~\ref{thrm:inverse_reversal} does not hold for arbitrary tensors $\mathcal{A} \in \mathbb{R}^{n_1 \times n_2 \times p}$ and $\mathcal{B} \in \mathbb{R}^{n_2 \times n_3 \times p}$. Simply, in general, 
    \begin{equation}\label{eq:neqInverseProperty}
        \left(\mathcal{A} \star_w \mathcal{B}\right)^{\dagger} \neq \mathcal{B}^{\dagger} \star_w \mathcal{A}^{\dagger}.
    \end{equation}   
\end{remark}
To justify this remark, let us take an example of tensors $\mathcal{A}\in \mathbb{R}^{2\times 3\times 2}$ and $\mathcal{B}\in \mathbb{R}^{3\times 2\times 2}$ such that,
{\small
$$\mathcal{A}^{(1)} = \begin{pmatrix}
0 & 1 & 1 \\
0 & 0 & 1
\end{pmatrix},~~ \mathcal{A}^{(2)} = \begin{pmatrix}
1 & 0 & 1 \\
0 & 1 & 1
\end{pmatrix},$$ 
$$\mathcal{B}^{(1)} = \begin{pmatrix}
0 & 0 \\
0 & 0 \\
1 & 1
\end{pmatrix},~~ \mathcal{B}^{(2)} = \begin{pmatrix}
0 & 0 \\
0 & 1 \\
1 & 0
\end{pmatrix},$$}

{\small
$$\mathcal{A}_{s} = \begin{pmatrix}
0.5 & 0.5 & 1.0 \\
0.0 & 0.5 & 1.0
\end{pmatrix},~~ \mathcal{A}_{d} = \begin{pmatrix}
-1 & 1 & 0 \\
0 & -1 & 0
\end{pmatrix},$$
$$\mathcal{B}_{s} = \begin{pmatrix}
0.0 & 0.0 \\
0.0 & 0.5 \\
1.0 & 0.5
\end{pmatrix},~~ \mathcal{B}_{d} = \begin{pmatrix}
0 & 0 \\
0 & -1 \\
0 & 1
\end{pmatrix},$$}

Now, computing left hand side of Eq.~\eqref{eq:neqInverseProperty}, which is $\left(\mathcal{A} \star_w \mathcal{B}\right)^{\dagger}$, assuming $\mathcal{C} = \left(\mathcal{A} \star_w \mathcal{B}\right)^{\dagger}$,

{\small
$$\left(\mathcal{A} \star_w \mathcal{B}\right)_{s} = \begin{pmatrix}
1.00 & 0.75 \\
1.00 & 0.75
\end{pmatrix},~~ \left(\mathcal{A} \star_w \mathcal{B}\right)_{d} = \begin{pmatrix}
0 & -1 \\
0 & 1
\end{pmatrix},$$
$$\mathcal{C}_{s} = \begin{pmatrix}
0.32 & 0.32 \\
0.24 & 0.24
\end{pmatrix},~~ \mathcal{C}_{d} = \begin{pmatrix}
0.0 & 0.0 \\
-0.5 & 0.5
\end{pmatrix},$$
$$\mathcal{C}^{(1)} = \begin{pmatrix}
0.32 & 0.32 \\
-0.01 & 0.49
\end{pmatrix},~~ \mathcal{C}^{(2)} = \begin{pmatrix}
0.32 & 0.32 \\
0.49 & -0.01
\end{pmatrix}.$$}

Now, computing right hand side of Eq.~\eqref{eq:neqInverseProperty}, which is $\mathcal{B}^{\dagger} \star_w \mathcal{A}^{\dagger}$, assuming $\mathcal{D} = \mathcal{B}^\dagger \star_w \mathcal{A}^\dagger$,
{\small
$$\mathcal{A}^\dagger_{s} = \begin{pmatrix}
2.0 & -2.0 \\ 
0.0 & 0.4 \\
0.0 & 0.8
\end{pmatrix},~~ \mathcal{A}^\dagger_{d} = \begin{pmatrix}
-1 & -1\\ 
0 & -1\\
0 & 0
\end{pmatrix},$$
$$\mathcal{B}^\dagger_{s} = \begin{pmatrix}
0 & -1 & 1\\
0 & 2 & 0
\end{pmatrix},~~ \mathcal{B}^\dagger_{d} = \begin{pmatrix}
0.0 & 0.0 & 0.0 \\
0.0 & -0.5 & 0.5
\end{pmatrix},$$
$$\mathcal{D}_{s} = \begin{pmatrix}
0.0 & 0.4 \\
0.0 & 0.8
\end{pmatrix},~~ \mathcal{D}_{d} = \begin{pmatrix}
0.0 & 0.0 \\
0.0 & 0.5
\end{pmatrix},$$
$$\mathcal{D}^{(1)} = \begin{pmatrix}
0.00 & 0.40 \\
0.00 & 1.05
\end{pmatrix},~~ \mathcal{D}^{(2)} = \begin{pmatrix}
0.00 & 0.40 \\
0.00 & 0.55
\end{pmatrix},$$}
As it is clear from above that $\mathcal{C}\neq\mathcal{D}$ which implies, $\left(\mathcal{A} \star_w \mathcal{B}\right)^{\dagger} \neq \mathcal{B}^{\dagger} \star_w \mathcal{A}^{\dagger}$. Hence, is the Remark~\ref{rem:neqInverseRev}.

\begin{lemma}[Properties of Moore-Penrose Inverse]
    Some properties of the Moore-Penrose inverse of tensor $\mathcal{A}\in\mathbb{R}^{m\times n\times p}$ are,
    \begin{enumerate}
        \item $(\mathcal{A}^\dagger)^\dagger = \mathcal{A}$
        \item $(\mathcal{A}^\top)^\dagger = (\mathcal{A}^\dagger)^\top$
        \item $(\mathcal{A}\star_w\mathcal{A}^\top)^\dagger = (\mathcal{A}^\top)^\dagger\star_w\mathcal{A}^\dagger$
        \item $\mathcal{A}^\dagger = \mathcal{A}^\top\star_w(\mathcal{A}\star_w\mathcal{A}^\top)^\dagger$
    \end{enumerate}
\end{lemma}
\begin{proof}
    These properties directly follow from the Definitions, Lemmas, and Theorems defined earlier in this Section.
\end{proof}

\begin{definition}[Trace]\label{def:trace}
    The trace of the tensor $\mathcal{A}\in \mathbb{R}^{n\times n\times p}$ is defined as,
    \begin{equation}
        \mathrm{Tr}(\mathcal{A}) = \sum_{k=1}^p\sum_{i=1}^n\mathcal{A}[i,i,k]
    \end{equation}
\end{definition}

\begin{lemma}[Properties of Trace]
    Some of the properties of the trace of a tensors $\mathcal{A},\mathcal{B}\in\mathbb{R}^{n\times n\times p}$ under $w$-product are,
    \begin{enumerate}
        \item $\mathrm{Tr}(\mathcal{A}) = \mathrm{Tr}(\mathcal{A}^\top)$.
        \item $\mathrm{Tr}(\alpha\mathcal{A} + \beta\mathcal{B}) = \alpha\mathrm{Tr}(\mathcal{A}) + \beta\mathrm{Tr}(\mathcal{B})$ where, $\alpha,\beta\in\mathbb{R}$.
        \item $\mathrm{Tr}(\mathcal{A}\star_w\mathcal{B}) = \mathrm{Tr}(\mathcal{B}\star_w\mathcal{A})$.
        \item $\mathrm{Tr}(\mathcal{A}^\top\star_w\mathcal{B}) = \mathrm{Tr}(\mathcal{A}\star_w\mathcal{B}^\top)$.
        \item $\mathrm{Tr}(\mathcal{A}^\dagger\star_w\mathcal{B}\star_w\mathcal{A}) = \mathrm{Tr}(\mathcal{B})$ if either $\mathcal{A}\star_w\mathcal{A}^\dagger\star_w\mathcal{B} = \mathcal{B}$ or, $\mathcal{B}\star_w\mathcal{A}\star_w\mathcal{A}^\dagger = \mathcal{B}$.
    \end{enumerate}
\end{lemma}
\begin{proof}
    Properties $1$ and $2$ follows directly from the definition of the trace and transpose of the tensors. 
    
    For property $3$, assume tensors $\mathcal{A},\mathcal{B}\in\mathbb{R}^{n\times n\times p}$ where $p = 2^L m$ and $m, L \in \mathbb{N}$. From Definition of $w$-product, $\mathcal{A}$ and $\mathcal{B}$ can be written as,
    $$W_{[L]}(\mathcal{A}) = \left(\mathcal{A}_{s_{L}}, \mathcal{A}_{d_{L}}, \mathcal{A}_{d_{(L-1)}}, \dots, \mathcal{A}_{d_{1}}\right)$$
    $$W_{[L]}(\mathcal{B}) = \left(\mathcal{B}_{s_{L}}, \mathcal{B}_{d_{L}}, \mathcal{B}_{d_{(L-1)}}, \dots, \mathcal{B}_{d_{1}}\right)$$
    Hence, from linearity of inverse wavelet transformation given in Lemma~\ref{lem:linearity} and known fact that Trace is linear operator,
    \begin{equation*}
        \begin{split}
            \mathrm{Tr}(\mathcal{A}\star_w\mathcal{B}) &= W_{[L]}^{-1}\left(\mathrm{Tr}\left(W_{[L]}(\mathcal{A})\Delta W_{[L]}(\mathcal{B})\right)\right)\\
            &= W_{[L]}^{-1}\left( \sum_{i=1}^m \mathrm{Tr}\left(\mathcal{A}_{s_L}[:,:,i]\mathcal{B}_{s_L}[:,:,i]\right)\right. +\\ &\qquad\qquad\left.\sum_{j=1}^L\sum_{k=1}^{p/2^j}\mathrm{Tr}\left(\mathcal{A}_{d_j}[:,:,k]\mathcal{B}_{d_j}[:,:,k]\right) \right)\\
            &=W_{[L]}^{-1}\left(\mathrm{Tr}\left(W_{[L]}(\mathcal{B})\Delta W_{[L]}(\mathcal{A})\right)\right)\\
            &= \mathrm{Tr}(\mathcal{B}\star_w\mathcal{A})
        \end{split}
    \end{equation*}
    Hence, property $3$ is proved. 
    
    Further, properties $4$ and $5$ are the direct consequences of the properties $1$-$3$ and Theorems defined earlier in this section.
\end{proof}


Singular Value Decomposition (SVD) described in Theorem~\ref{thrm:svd} for a tensor $\mathcal{A} \in \mathbb{R}^{n_1 \times n_2 \times p}$, where $p = 2^L m$ and $m, L \in \mathbb{N}$, can be efficiently computed using Algorithm~\ref{alg:svd_w_product}.

\begin{algorithm}
\small
\caption{SVD using $w$-product for tensors}
\label{alg:svd_w_product}
\begin{algorithmic}[1]
\REQUIRE Tensor $\mathcal{A} \in \mathbb{R}^{n_1 \times n_2 \times p}$ with $p=2^L m$ where $m,L \in \mathbb{N}$ 

\STATE Initialize: $\mathcal{A}_{s_{0}} = \mathcal{A}$
\FOR{$j=1$ {\textbf{to}} $L$}
    \STATE Set: $\ell = p/\left(2^{j-1}\right)$
    \STATE Detail coefficients: $\mathcal{A}_{d_{j}} = \mathcal{A}_{s_{j-1}}[:,:,1:2:(\ell-1)] - \mathcal{A}_{s_{j-1}}[:,:,2:2:\ell]$
    \STATE Smooth coefficients: $\mathcal{A}_{s_{j}} = \mathcal{A}_{s_{j-1}}[:,:,2:2:\ell] + \mathcal{A}_{d_{j}} / 2$
\ENDFOR

\FOR{$j=1$ {\textbf{to}} $L$}
    \FOR{$k=1$ {\textbf{to}} $p/2^{j}$}
        \STATE $U_{d_{j}}, \Sigma_{d_{j}}, V_{d_{j}} \gets \mathrm{svd}\left(\mathcal{A}_{d_j}[:,:,k]\right)$
        \STATE $\mathcal{C}_{d_{j}}[:,:,k] \gets U_{d_{j}}[:,:r]\Sigma_{d_{j}}[:r,:r] \left(V_{d_{j}}[:,:r]\right)^\top$
        \IF{$j$ {\textbf{is}} $L$}
            \STATE $U_{s_L}, \Sigma_{s_L}, V_{s_L} \gets \mathrm{svd}\left(\mathcal{A}_{s_L}[:,:,k]\right)$
            \STATE $\mathcal{C}_{s_{L}}[:,:,k] \gets U_{s_L}[:,:r]\Sigma_{s_L}[:r,:r] \left(V_{s_L}[:,:r]\right)^\top$
        \ENDIF
    \ENDFOR
\ENDFOR

\STATE Initialize: $\mathcal{C} \in \mathbb{R}^{n_1\times n_3 \times p}$
\FOR{$j = L$ {\textbf{downto}} $1$}
    \STATE Set: $\ell = p/\left(2^{j-1}\right)$
    \STATE $\mathcal{C}_{s_{j-1}}[:,:,2:2:\ell] = \mathcal{C}_{s_j} - \mathcal{C}_{d_j} / 2$
    \STATE $\mathcal{C}_{s_{j-1}}[:,:,1:2:(\ell-1)] = \mathcal{C}_{d_j} + \mathcal{C}_{s_{j-1}}[:,:,2:2:\ell]$
\ENDFOR
\STATE Finalize: $\mathcal{C} = \mathcal{C}_{s_0}$

\RETURN $\mathcal{C}$
\end{algorithmic}
\end{algorithm}

\section{Complexity Analysis and Discussions}\label{sec:complexity}
For the multiplication of tensors $\mathcal{A} \in \mathbb{R}^{n_1 \times n_2 \times p}$ and $\mathcal{B} \in \mathbb{R}^{n_2 \times n_3 \times p}$ using the $t$-product~\cite{kilmer2011}, denoted $\mathcal{A} \star_t \mathcal{B}$, and the $m$-product~\cite{kernfeld2015}, denoted $\mathcal{A} \star_m \mathcal{B}$, the computational complexities vary significantly for large tensors. The complexity of the $m$-product is, 
$\mathcal{O}(n_1 n_2 n_3 p)$ $+$ $\mathcal{O}\left((n_1 n_2 + n_2 n_3 + n_1 n_3) p^2\right),$
whereas the complexity of the $t$-product is,
$\mathcal{O}(n_1 n_2 n_3 p)$ $+$ $\mathcal{O}\left((n_1 n_2 + n_2 n_3 + n_1 n_3) p \log p\right).$
The proposed $w$-product, aims to reduce the complexity incurred during FFT transformation by incorporating wavelet transformation, which has a complexity order $\mathcal{O}(2n)$ compared to FFT’s order $\mathcal{O}(n \log n)$. For the worst case, choosing $\log_2 p \in \mathbb{N}$, the complexity of the $w$-product is given by,
$\mathcal{O}(n_1 n_2 n_3 p)$ $+$ $\mathcal{O}\left(2p (n_1 n_2 + n_2 n_3 + n_1 n_3)\right).$
This results in a significant reduction in the second term compared to FFT-based products, especially for large $p$, due to the linear complexity growth in wavelet transforms.

For the algorithmic comparison, choosing $n_1 = n_2 = n_3 = p$, the complexities simplify as follows,
\begin{enumerate}
    \item The complexity of $\mathcal{A} \star_m \mathcal{B}$ is $\mathcal{O}(p^4 + 3 p^4)$.
    \item The complexity of $\mathcal{A} \star_t \mathcal{B}$ is $\mathcal{O}(p^4 + 3 p^3 \log p)$.
    \item The complexity of $\mathcal{A} \star_w \mathcal{B}$ is $\mathcal{O}(p^4 + 6 p^3)$.
\end{enumerate}
A comparison of the analytical complexities of $m$-product, $t$-product, and $w$-product is illustrated in Fig.~\ref{fig:complexity}.

The time complexity analysis was conducted using an Intel Core $i9-12900$K CPU ($3.20$GHz base, up to $5.20$GHz Turbo), $32$GB RAM, and $16$ cores. Experiments were performed using Python $3.11.0$. 

For experimental time analysis, we assume tensors $\mathcal{A}, \mathcal{B}\in\mathbb{R}^{p\times p\times p}$ which are created using \texttt{torch.rand}. All the experimental results in Table~\ref{tab:computation_times} were averaged over $10$ runs. Fig.~\ref{fig:time} shows these results, which follow the same trend as the expected analytical results shown in Fig.~\ref{fig:complexity}.

\begin{table}
\centering
\caption{Average computation times over $10$ runs for different product methods over varying $p$ for multiplying $\mathcal{A}\in\mathbb{R}^{p\times p\times p}$ and $\mathcal{B}\in\mathbb{R}^{p\times p\times p}$.}
\label{tab:computation_times}
\resizebox{\linewidth}{!}{
\begin{tabular}{|c|c|c|c|}
\hline
$p$ & $m$-product (sec) & $t$-product (sec) & $w$-product (sec) \\ \hline
$2^1$ & 9.466076e-04 & 2.240753e-04 & 1.576734e-04 \\ \hline
$2^2$ & 9.183884e-04 & 3.964186e-04 & 1.577830e-04 \\ \hline
$2^3$ & 2.738118e-03 & 4.737735e-04 & 3.587484e-04 \\ \hline
$2^4$ & 9.937849e-03 & 6.299186e-04 & 4.789996e-04 \\ \hline
$2^5$ & 8.881669e-02 & 2.113118e-03 & 1.763617e-03 \\ \hline
$2^6$ & 1.641816e-01 & 1.014814e-02 & 8.662295e-03 \\ \hline
$2^7$ & 4.362548e+00 & 7.131300e-02 & 5.306451e-02 \\ \hline
$2^8$ & 2.615309e+01 & 5.213821e-01 & 3.919755e-01 \\ \hline
$2^9$ & 1.123852e+02 & 4.842443e+00 & 3.161437e+00 \\ \hline
$2^{10}$ & 6.378374e+02 & 1.574944e+02 & 7.749125e+01 \\ \hline
\end{tabular}}
\end{table}

\begin{figure}
    \centering
    \includegraphics[width=0.75\linewidth]{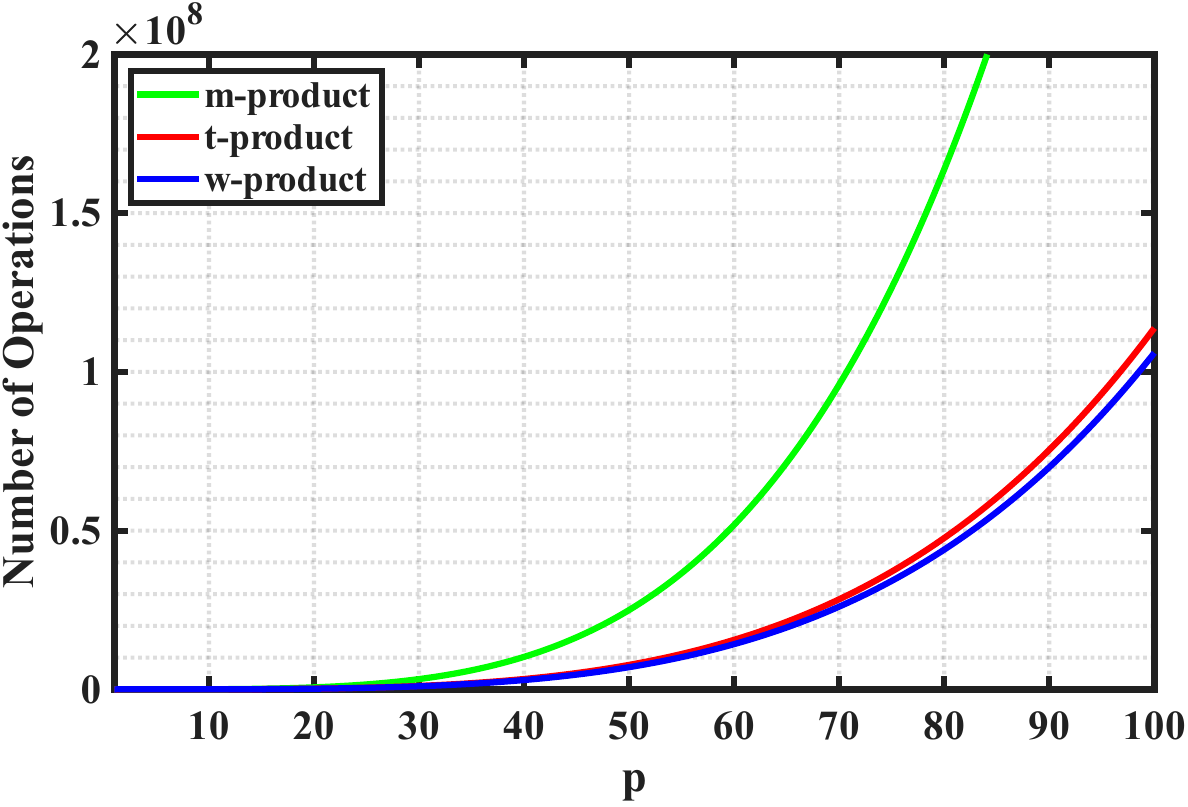}
    \caption{Analytical complexity (Number of operations).}
    \label{fig:complexity}
\end{figure}

\begin{figure}
    \centering
    \includegraphics[width=0.75\linewidth]{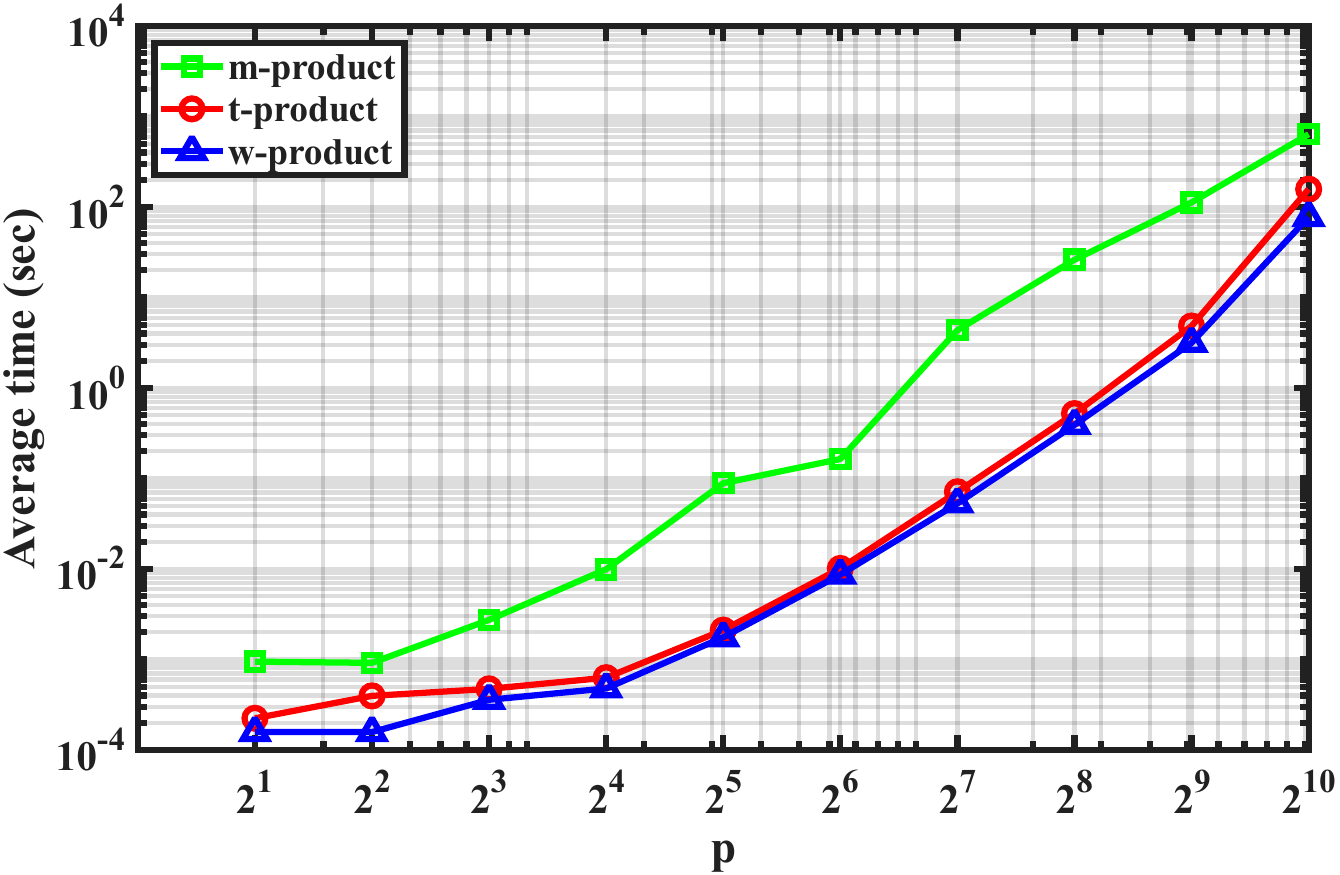}
    \caption{Experimental time taken in multiplication.}
    \label{fig:time}
\end{figure}

\begin{remark}
    One can extend the $w$-product by selecting specific prediction and update operators at any level according to the problem requirements. This flexibility in designing the operators at each scale enables customization and adaptation of the wavelet transform, making the $w$-product highly versatile and useful for various applications.
\end{remark}


\section{Applications to Hyperspectral Image Analysis}\label{sec:application}
Hyperspectral images capture detailed spectral information across many narrow, contiguous wavelength bands for each pixel, allowing precise identification of materials and objects based on their unique spectral signatures. This creates a tensor that combines spatial and spectral information, which is useful in many applications, including remote sensing.

\subsection{Experimental setup}
All comparisons are performed using Python $3.11.0$. For performance comparisons, we used the Peak Signal-to-Noise Ratio (PSNR) and Structural Similarity Index (SSIM) for two tensors $\mathcal{X}_1, \mathcal{X}_2\in\mathbb{R}^{n_1\times n_2\times p}$. The PSNR value is given by,
\begin{equation}
    \text{PSNR } = 10\log_{10}\left(\frac{n_1\times n_2\times p\times\mathrm{MPP}}{\|\mathcal{X}_1-\mathcal{X}_2\|^2_F}\right),
\end{equation}
where MPP is the maximum possible pixel value on the tensors $\mathcal{X}_1$ and $\mathcal{X}_2$. Further, the SSIM value between the tensors $\mathcal{X}_1, \mathcal{X}_2$ is given by
\begin{equation}
    \text{SSIM}\left(\mathcal{X}_1, \mathcal{X}_2\right) = \frac{1}{p}\sum_{k=1}^{p} \mathrm{SSIM}\left(\mathcal{X}_1[:,:,k], \mathcal{X}_2[:,:,k]\right),
\end{equation}
where, for $\mu_{x_1},\mu_{x_2}$ as mean, $\sigma_{x_1}^2, \sigma_{x_2}^2$ as variance and $\sigma_{x_1x_2}$ as covariance, and $c_1, c_2$ as stabilizing constants, the SSIM $(x_1,x_2)$ is given by,
\begin{equation*}
    \text{SSIM}(x_1,x_2) = \frac{\left(2\mu_{x_1}\mu_{x_2} + c_1\right)\left(2\sigma_{x_1x_2} + c_2\right)}{\left(\mu_{x_1}^2 + \mu_{x_2}^2 + c_1\right)\left(\sigma_{x_1}^2 + \sigma_{x_2}^2 + c_2\right)}
\end{equation*}

\subsection{Application to low-rank hyperspectral image reconstruction}
\subsubsection{Computational analysis}
The tensor $\mathcal{A} \in \mathbb{R}^{n_1 \times n_2 \times p}$ is considered under the assumption that $\log_2 p = L \in \mathbb{N}$ and $n_1 = n_2 = p$. In this setting, the dominant informational content of $\mathcal{A}$ is concentrated in the matrices $\mathcal{A}_{s_L}$ and $\mathcal{A}_{d_L}$ due to the presence of redundant information in the hyperspectral image. When employing the $t$-product framework (denote ``$t$-svd''), the SVD exhibits a computational complexity of $\mathcal{O}(p^4 + 2p^3 \log p)$. In comparison, the corresponding complexities for the $m$-product (denote ``$m$-svd'') and $w$-product (denote ``$w$-svd'') formulations are $\mathcal{O}(p^4 + 2p^4)$ and $\mathcal{O}(p^4 + 4p^3)$, respectively. However, by leveraging the sparsity properties inherent to wavelet representations (leveraging use of only $\mathcal{A}_{s_L}$ and $\mathcal{A}_{d_L}$ matrices), the sparse $w$-product formulation (denote ``sp-$w$-svd'') reduces the overall complexity to $\mathcal{O}(2p^3 + 4p^3)$. The corresponding performance comparison and computational advantages are illustrated in Fig.~\ref{fig:spw-svd}. 

\begin{figure}
    \centering
    \includegraphics[width=0.75\linewidth]{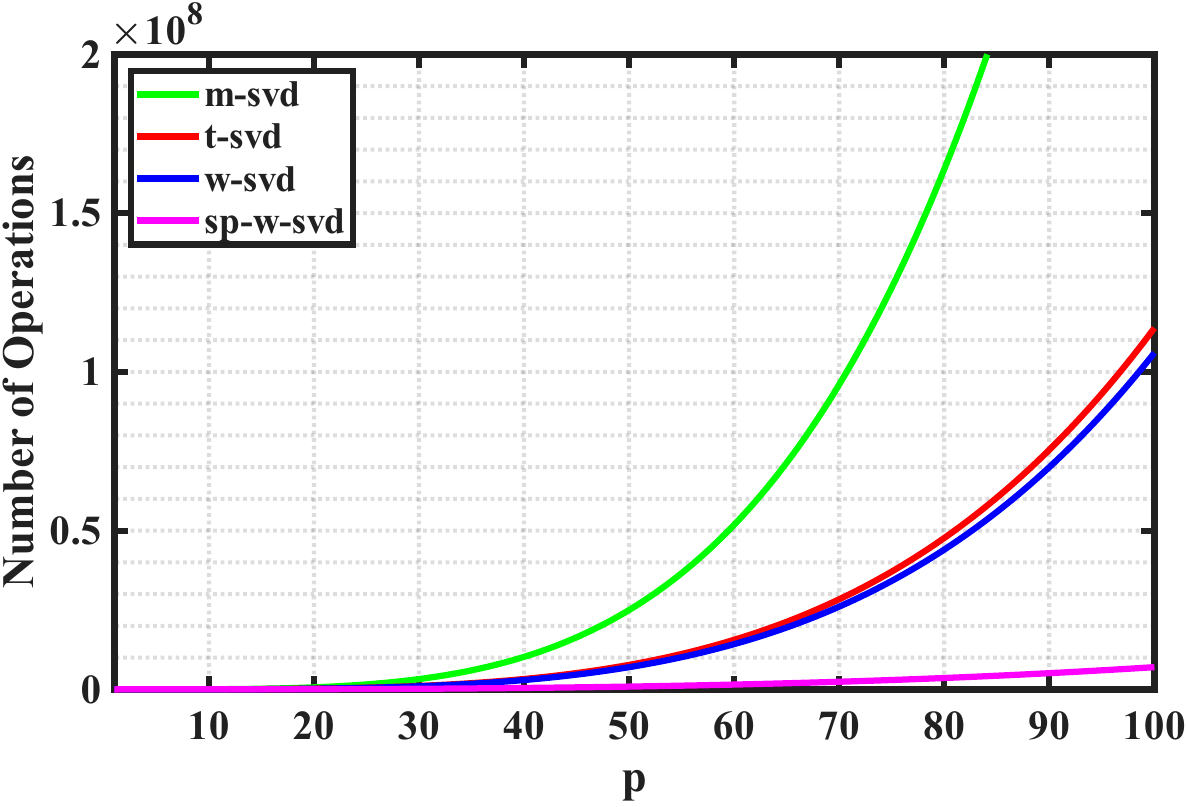}
    \caption{Number of operations required for SVD using various products.}
    \label{fig:spw-svd}
\end{figure}

\begin{remark}
    Proposed $w$-product structure plays a significant role in SVD reduction of tensor $\mathcal{A}\in\mathbb{R}^{n_1\times n_2\times p}$ with reducing the state-of-the-art complexity of ``identity-svd'' (used in majority of the libraries including \texttt{torch}) and ``$t$-svd'' of $\mathcal{O}(n_1n_2p\min\{n_1,n_2\})$ to $\mathcal{O}(n_1n_2\min\{n_1,n_2\})$.
\end{remark}

Fig.~\ref{fig:sp-svd-working} demonstrates the ``sp-$w$-svd'' reconstruction of the original hyperspectral rice image of size $256 \times 192$ with $96$ spectral bands. It can be easily seen that most of the information is retained by the smooth scale coefficients at every level, while the detailed information diminishes as we move upwards. Hence, reconstruction using the finest detail and smooth coefficients provides a good resolution of the image without sacrificing much information.

\begin{figure}
    \centering
    \includegraphics[width=0.9\linewidth]{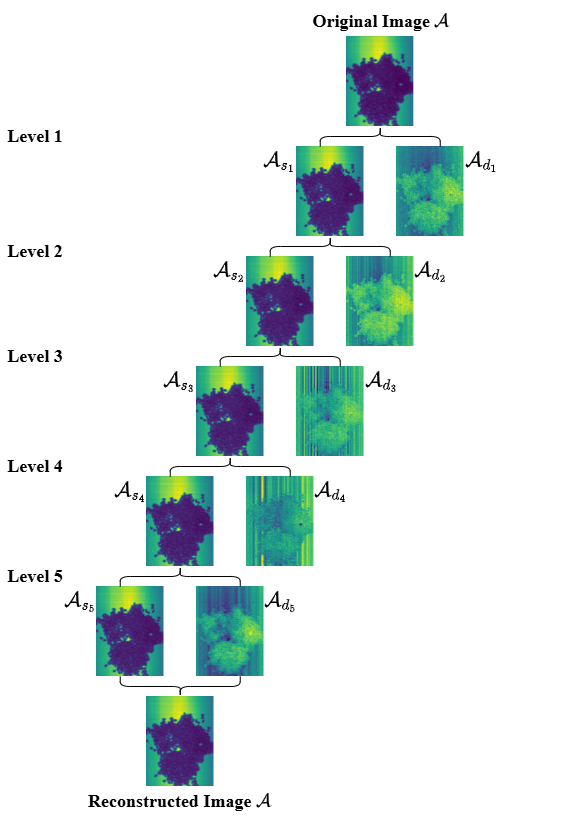}
    \caption{Sparse SVD (sp-$w$-svd) reconstruction using Rank $32$ decomposition of tensor $\mathcal{A}$ of size ${256\times 192\times 96}$.}
    \label{fig:sp-svd-working}
\end{figure}

\subsubsection{Experimental results}
We considered three hyperspectral images, coffee, yatsuhashi, and sugar\_salt\_flour, each of size $256 \times 192 \times 96$, that is, each consisting of 96 distinct spectral bands. The averaged spectral image is treated as the full-rank representation of the corresponding hyperspectral data. For the third dimension of size $96$, both ``$w$-svd'' and ``sp-$w$-svd'' perform a $5$-level wavelet decomposition. Table~\ref{tab:decom_figs} illustrates the SVD decompositions of hyperspectral images at various ranks using the ``sp-$w$-svd'' method. Table~\ref{tab:applications} presents the numerical comparison of PSNR, SSIM, and execution time among the ``$t$-svd'', ``$w$-svd'', and ``sp-$w$-svd'' techniques. It can be observed that ``$w$-svd'' achieves up to $5.91$ times speedup and ``sp-$w$-svd'' provides up to $92.21$ times speedup compared to ``$t$-svd'', while preserving comparable PSNR and SSIM performance.

\begin{table}
\caption{SVD comparison of various ranks for image decompositions.}
\label{tab:decom_figs}
\centering
\resizebox{\linewidth}{!}{
\begin{tabular}{l c c c c c c c}
\hline
Image & Full Rank & Rank $2$ & Rank $4$ & Rank $8$ & Rank $16$ & Rank $32$ & Rank $64$ \\\hline
Coffee & \includegraphics[width = 0.075\textwidth]{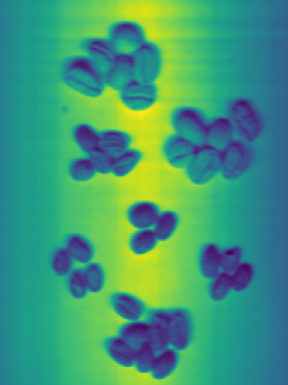} & \includegraphics[width = 0.075\textwidth]{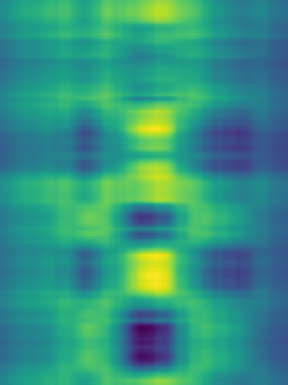} & \includegraphics[width = 0.075\textwidth]{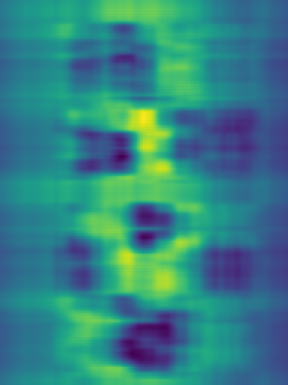} & \includegraphics[width = 0.075\textwidth]{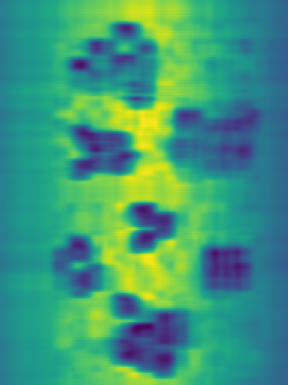} & \includegraphics[width = 0.075\textwidth]{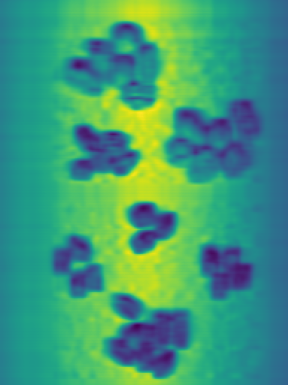} & \includegraphics[width = 0.075\textwidth]{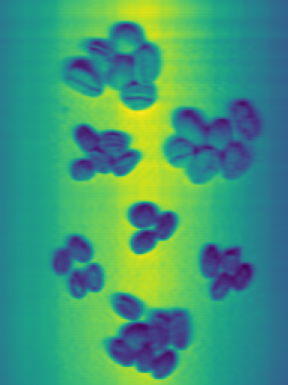} & \includegraphics[width = 0.075\textwidth]{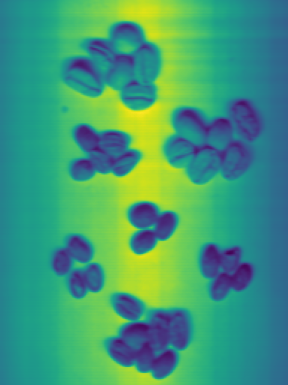}\\\hline
Sugar-Salt-Flour & \includegraphics[width = 0.075\textwidth]{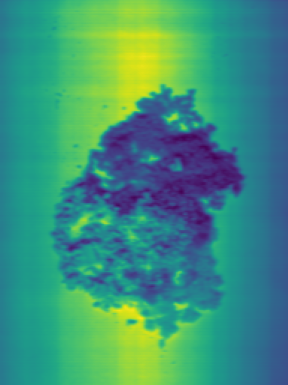} & \includegraphics[width = 0.075\textwidth]{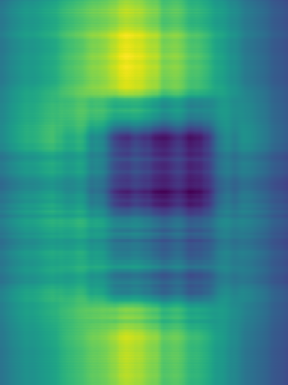} & \includegraphics[width = 0.075\textwidth]{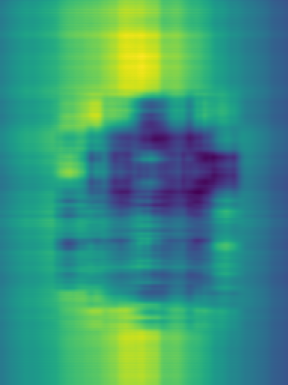} & \includegraphics[width = 0.075\textwidth]{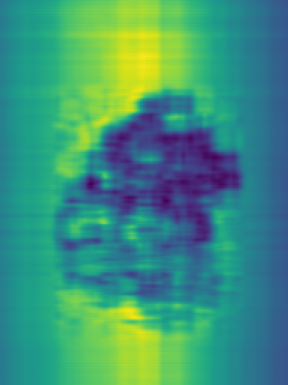} & \includegraphics[width = 0.075\textwidth]{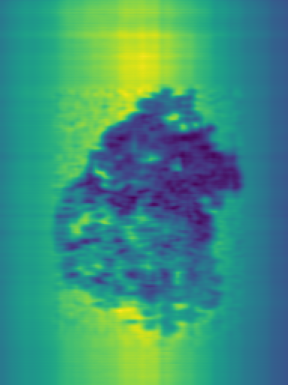} & \includegraphics[width = 0.075\textwidth]{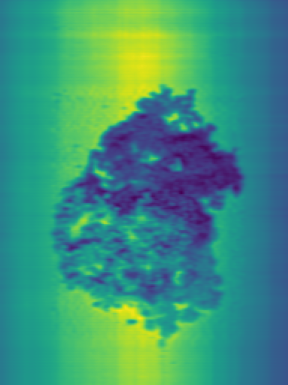} & \includegraphics[width = 0.075\textwidth]{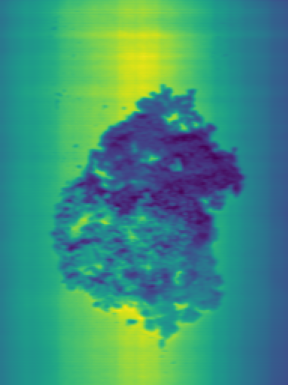}\\\hline
Yatsuhashi & \includegraphics[width = 0.075\textwidth]{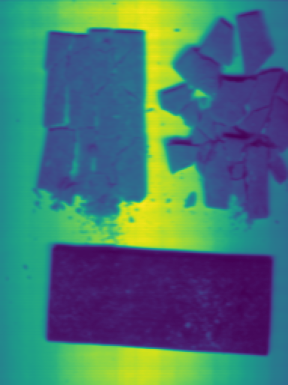} & \includegraphics[width = 0.075\textwidth]{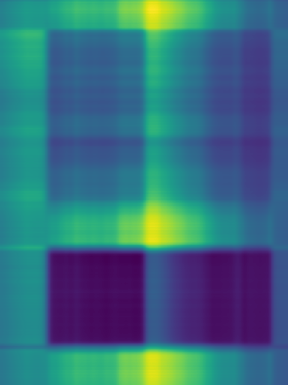} & \includegraphics[width = 0.075\textwidth]{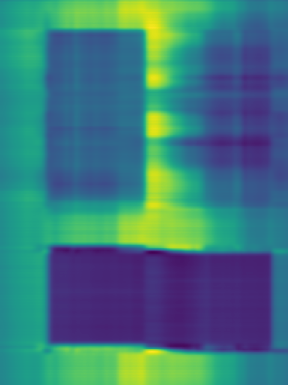} & \includegraphics[width = 0.075\textwidth]{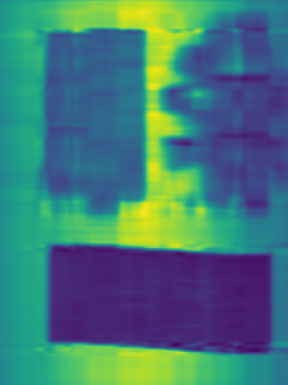} & \includegraphics[width = 0.075\textwidth]{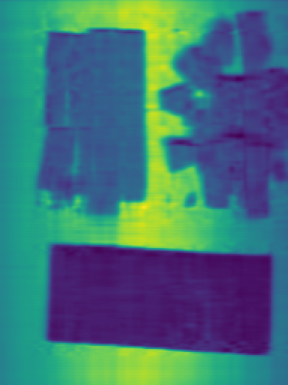} & \includegraphics[width = 0.075\textwidth]{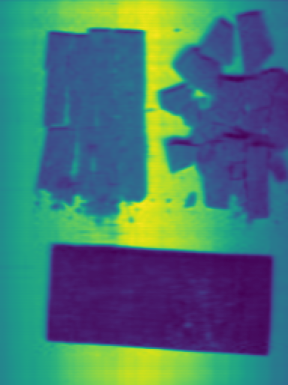} & \includegraphics[width = 0.075\textwidth]{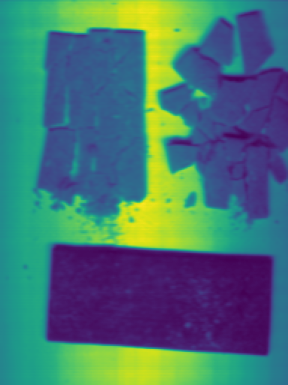}
\\\hline
\end{tabular}}
\end{table}

\begin{table*}
\centering
\caption{Results across $3$ different hyperspectral images of sizes $256\times 192\times 96$.}
\label{tab:applications}
\resizebox{\linewidth}{!}{%
\begin{tabular}{|c|c|c|c|c|c|c|c|c|c|c|c|c|c|}
\hline
\multirow{2}{*}{\textbf{Dataset}} & \multirow{2}{*}{\textbf{Rank}} & \multicolumn{4}{c|}{\textbf{$t$-svd}}  & \multicolumn{4}{c|}{\textbf{$w$-svd}} & \multicolumn{4}{c|}{\textbf{sp-$w$-svd}} \\\cline{3-14}
 & & SSIM & PSNR & Time & Speedup & SSIM & PSNR & Time & Speedup & SSIM & PSNR & Time & Speedup \\
\hline
\multirow{6}{*}{Coffee} & 2 & 0.6266 & 18.0599 & 5.0527 & 1.00$\times$ & 0.6271 & 18.0696 & 0.8583 & 5.89$\times$ & 0.6266 & 18.0612 & 0.0548 & 92.21$\times$\\
 & 4 & 0.6688 & 21.1696 & 5.1189 & 1.00$\times$ & 0.6699 & 21.1863 & 0.8812 & 5.81$\times$ & 0.6691 & 21.1745 & 0.0762 & 67.18$\times$ \\
 & 8 & 0.7914 & 25.4913 & 5.1427 & 1.00$\times$ & 0.7925 & 25.5160 & 0.8808 & 5.84$\times$ & 0.7918 & 25.5002 & 0.0783 & 65.68$\times$ \\
 & 16 & 0.9076 & 32.0552 & 5.1580 & 1.00$\times$ & 0.9098 & 32.1351 & 0.8862 & 5.82$\times$ & 0.9088 & 32.0987 & 0.0689 & 74.86$\times$ \\
 & 32 & 0.9842 & 41.7228 & 5.1278 & 1.00$\times$ & 0.9856 & 42.0571 & 0.9026 & 5.68$\times$ & 0.9852 & 42.0228 & 0.0674 & 76.08$\times$ \\
 & 64 & 0.9992 & 56.8858 & 5.1108 & 1.00$\times$ & 0.9993 & 57.2451 & 0.8961 & 5.71$\times$ & 0.9992 & 57.1023 & 0.0802 & 63.73$\times$ \\
\hline
\multirow{6}{*}{Sugar-Salt-Flour} & 2 & 0.7040 & 21.2944 & 5.1041 & 1.00$\times$ & 0.7048 & 21.3123 & 0.8642 & 5.91$\times$ & 0.7041 & 21.2964 & 0.0631 & 80.89$\times$ \\
 & 4 & 0.7595 & 23.6354 & 5.1006 & 1.00$\times$ & 0.7665 & 23.8537 & 0.8866 & 5.75$\times$ & 0.7620 & 23.7136 & 0.0757 & 67.38$\times$ \\
 & 8 & 0.8400 & 27.5823 & 5.1191 & 1.00$\times$ & 0.8490 & 27.9232 & 0.8871 & 5.77$\times$ & 0.8437 & 27.7379 & 0.0730 & 70.12$\times$ \\
 & 16 & 0.9286 & 33.0338 & 5.1793 & 1.00$\times$ & 0.9338 & 33.3231 & 0.8912 & 5.81$\times$ & 0.9301 & 33.1185 & 0.0805 & 64.34$\times$ \\
 & 32 & 0.9833 & 41.4049 & 5.1361 & 1.00$\times$ & 0.9856 & 41.8537 & 0.8968 & 5.73$\times$ & 0.9842 & 41.5716 & 0.0802 & 64.04$\times$ \\
 & 64 & 0.9988 & 55.6019 & 5.1562 & 1.00$\times$ & 0.9990 & 56.3855 & 0.9006 & 5.73$\times$ & 0.9990 & 56.0211 & 0.0732 & 70.44$\times$ \\
\hline
\multirow{6}{*}{Yatsuhashi} & 2 & 0.6090 & 18.6451 & 5.0454 & 1.00$\times$ & 0.6165 & 18.7917 & 0.8646 & 5.84$\times$ & 0.6121 & 18.6984 & 0.0641 & 78.71$\times$ \\
 & 4 & 0.7282 & 22.9088 & 5.0624 & 1.00$\times$ & 0.7292 & 22.9382 & 0.8892 & 5.69$\times$ & 0.7285 & 22.9169 & 0.0812 & 62.34$\times$ \\
 & 8 & 0.8055 & 27.5038 & 5.1430 & 1.00$\times$ & 0.8076 & 27.5739 & 0.8898 & 5.78$\times$ & 0.8058 & 27.5171 & 0.0795 & 64.69$\times$ \\
 & 16 & 0.8960 & 33.0988 & 5.1076 & 1.00$\times$ & 0.8991 & 33.2449 & 0.8893 & 5.74$\times$ & 0.8971 & 33.1488 & 0.0802 & 63.69$\times$ \\
 & 32 & 0.9685 & 40.6619 & 5.1695 & 1.00$\times$ & 0.9707 & 40.9473 & 0.8980 & 5.76$\times$ & 0.9697 & 40.8092 & 0.0796 & 64.94$\times$ \\
 & 64 & 0.9970 & 52.1228 & 5.1286 & 1.00$\times$ & 0.9973 & 52.6252 & 0.9035 & 5.68$\times$ & 0.9972 & 52.3533 & 0.0753 & 68.11$\times$ \\
\hline
\end{tabular}}
\end{table*}

\subsection{Application to hyperspectral image deblurring}
\subsubsection{Computational analysis}

Given an observed blurred hyperspectral image tensor $\mathcal{B}\in\mathbb{R}^{n_1 \times n_2 \times p}$, the deblurring problem can be expressed as $\mathcal{X} = \mathcal{A}^\dagger \star \mathcal{B}$, where $\mathcal{X}\in\mathbb{R}^{n_1 \times n_2 \times p}$ is the required true hyperspectral image tensor and $\mathcal{A}\in\mathbb{R}^{n_1 \times n_1 \times p}$ is the tensor blurring operator.
Under the assumptions $\log_2 p = L \in \mathbb{N}$ and $n_1 = n_2 = p$, it is observed that the dominant informational content of $\mathcal{X}$ is concentrated primarily in the matrices $\mathcal{X}_{s_L}$ and $\mathcal{X}_{d_L}$. It is known that the computation of the pseudo-inverse can be performed using SVD, where the computational complexity for a matrix $Z \in \mathbb{R}^{m \times n}$ is $\mathcal{O}(mn \min(m, n))$. For conventional $m$-product reconstruction, the computational complexity of deblurring is $\mathcal{O}(2p^4 + 3p^4)$. In contrast, the $t$-product and $w$-product reconstructions exhibit reduced complexities of $\mathcal{O}(2p^4 + 3p^3 \log p)$ and $\mathcal{O}(2p^4 + 6p^3)$, respectively. However, by leveraging the properties of hyperspectral images and wavelet transformations, the sparse $w$-product (sp-$w$-product) formulation further reduces the overall complexity of reconstruction to $\mathcal{O}(4p^3 + 6p^3)$.

\begin{remark}\label{hyperspectral_redundant}
    Due to highly correlated spectral information and significant redundancy in hyperspectral remote sensing images, there is very little variation in pixel-wise information. Due to this, most of the detailed components of the wavelet decomposition tend to be nearly zero~\cite{luo2021redundant, liao2023redundant}.
\end{remark}

\subsubsection{Experimental results}
Consider an input tensor $\mathcal{X} \in \mathbb{R}^{n_1 \times n_2 \times p}$ representing a hyperspectral image of spatial size $n_1 \times n_2$ with $p$ spectral bands. For the blurring operation, vertical and horizontal linear blurring are assumed. Given vertical blur size $b_v$ and horizontal blur size $b_h$, we define the vertical and horizontal blur vectors as,
\begin{equation*}
    c_v(i) = \begin{cases}
        \frac{b_v - i}{3 b_v}, & 0 \leq i < b_v \\
        0, & \text{otherwise}
    \end{cases}, \quad i = 0, 1, \ldots, n_1-1,
\end{equation*}
\begin{equation*}
    c_h(j) = \begin{cases}
        \frac{b_h - j}{3 b_h}, & 0 \leq j < b_h \\
        0, & \text{otherwise}
    \end{cases}, \quad j = 0, 1, \ldots, n_2-1.
\end{equation*}

Accordingly, we define Toeplitz matrices $A_v \in \mathbb{R}^{n_1 \times n_1}$ and $A_h \in \mathbb{R}^{n_2 \times n_2}$ with entries defined as $(A_v)_{ij} = c_v(|i-j|)$ and $(A_h)_{ij} = c_h(|i-j|)$. Going forward, we define multiband blur tensors $\mathcal{A_V}$ and $\mathcal{A_H}$ such that for each $k=1,2,\dots,p$, 
\begin{equation*}
    \mathcal{A_V}[:, :, k] = A_v, \quad \mathcal{A_H}[:, :, k] = A_h.
\end{equation*}

Then, the blurring operation on the hyperspectral image $\mathcal{X}$ is defined as,
\begin{equation}
    \mathcal{B} = \mathcal{A_V} \star_w \mathcal{X} \star_w \mathcal{A_H}^\top.
\end{equation}

The deblurring operation, which provides the least squares solution to extract the original hyperspectral image from the blurred image using the known blurring operators, is given by,
\begin{equation}
    \mathcal{X} = \mathcal{A_V}^\dagger \star_w \mathcal{B} \star_w \left(\mathcal{A_H}^\top\right)^\dagger,
\end{equation}

For hyperspectral images, we have used the remote sensing image dataset provided by Acerca de Grupo de Inteligencia Computacional (GIC)~\cite{dataset2}. From various available hyperspectral images, we consider three different images (\texttt{PaviaU}, \texttt{Salinas}, and \texttt{Botswana}) to show the efficiency of the $w$-product over existing products,
\begin{enumerate}
    \item[i).] Pavia University: Acquired by the ROSIS sensor over Pavia, Italy, with spatial dimensions $610 \times 340$ pixels and $102$ spectral bands. The geometric resolution is $1.3$m.
    \item[ii).] Salinas: Captured by the AVIRIS sensor over Salinas Valley, California, originally with $224$ bands; after removing $20$ water absorption bands, $204$ bands remain. The spatial resolution is $3.7$m.
    \item[iii).] Botswana: Collected by NASA EO-1 Hyperion sensor over the Okavango Delta in $2001-2004$. The data has $30$m spatial resolution and $144$ spectral bands after removing uncalibrated and noisy water absorption bands.
\end{enumerate}

We considered three hyperspectral images, \texttt{PaviaU} of size $610\times 340\times 102$, \texttt{Salinas} of size $512\times 217\times 204$, and \texttt{Botswana} of size $1476\times 256\times 144$. For linear blurring tensors $\mathcal{A_V}$ and $\mathcal{A_H}$, we assume values $b_v = b_h = 10$. For \texttt{PaviaU} with $102$ spectral bands, $w$-product reconstruction and sp-$w$-product reconstruction perform $1$-level wavelet decomposition. For \texttt{Salinas} and \texttt{Botswana} with $204$ and $144$ spectral bands, $w$-product reconstruction and sp-$w$-product reconstruction perform $2$-level and $4$-level wavelet decomposition, respectively. Table~\ref{tab:deblur_figs} illustrates the deblurring results of the hyperspectral images using the ``$t$-product'', ``$w$-product'', and ``sp-$w$-product'' reconstruction methods. Table~\ref{tab:results_deblur} presents the numerical comparison of PSNR, SSIM, and execution time among the ``$t$-product'', ``$w$-product'', and ``sp-$w$-product'' reconstruction techniques. 

From Table~\ref{tab:results_deblur}, it can be observed that ``sp-$w$-product'' reconstruction achieves up to $27.88$ times speedup compared to ``$t$-product'' reconstruction, along with improvement in PSNR and SSIM values. It can also be observed that the $w$-product and sp-$w$-product reconstructions have the same PSNR and SSIM, which is due to the fact that hyperspectral images have a large amount of redundant information (given by Remark~\ref{hyperspectral_redundant}). Also, it can be seen that as the number of decomposition levels increases from $1$ to $4$, the speedup achieved by $w$-product and sp-$w$-product reconstructions over ``$t$-product'' reconstruction increases exponentially.

\begin{table}[]
\centering
\caption{Hyperspectral Image Deblurring Performance across Methods and Images.}
\label{tab:results_deblur}
\resizebox{\linewidth}{!}{%
\begin{tabular}{|c|c|c|c|c|}
\hline
\textbf{Image ($w$-levels)} & \textbf{Method} & \textbf{Time (sec)} & \textbf{SSIM} & \textbf{PSNR} \\\hline
\multirow{4}{*}{} 
& $t$-product & 23.6328 & 0.4201 & 21.1510 \\
PaviaU ($1$ level) & $w$-product & 9.8508 & 0.6992 & 26.7206 \\
$610\times 340\times 102$ & sp-$w$-product & 9.7492 & 0.6992 & 26.7206 \\ 
& Speedups & $2.42\times$ & $-$ & $-$ \\
\hline
\multirow{4}{*}{} 
& $t$-product & 31.4715 & 0.2046 & 17.6921 \\ 
Salinas ($2$ levels) & $w$-product & 10.9848 & 0.5726 & 23.6444 \\
$512\times 217\times 204$ & sp-$w$-product & 6.5197 & 0.5726 & 23.6444 \\
& Speedups & $4.83\times$ & $-$ & $-$ \\
\hline
\multirow{4}{*}{} 
& $t$-product & 204.7620 & 0.7552 & 34.4265 \\ 
Botswana ($4$ levels) & $w$-product & 40.5728 & 0.8622 & 36.7636 \\
$1476\times 256\times 144$ & sp-$w$-product & 7.3428 & 0.8622 & 36.7636 \\
& Speedups & $27.88\times$ & $-$ & $-$ \\
\hline
\end{tabular}}
\end{table}

\begin{table*}
\caption{Hyperspectral Image Deblurring Reconstruction on Remote Sensing Data.}
\label{tab:deblur_figs}
\centering
\resizebox{\linewidth}{!}{
\begin{tabular}{l c c c c c}
\hline
Image & Original Image & Blurred Image & $t$-product Reconstructed & $w$-product Reconstructed & sp-$w$-product Reconstructed \\\hline

PaviaU & \includegraphics[width = 0.15\textwidth]{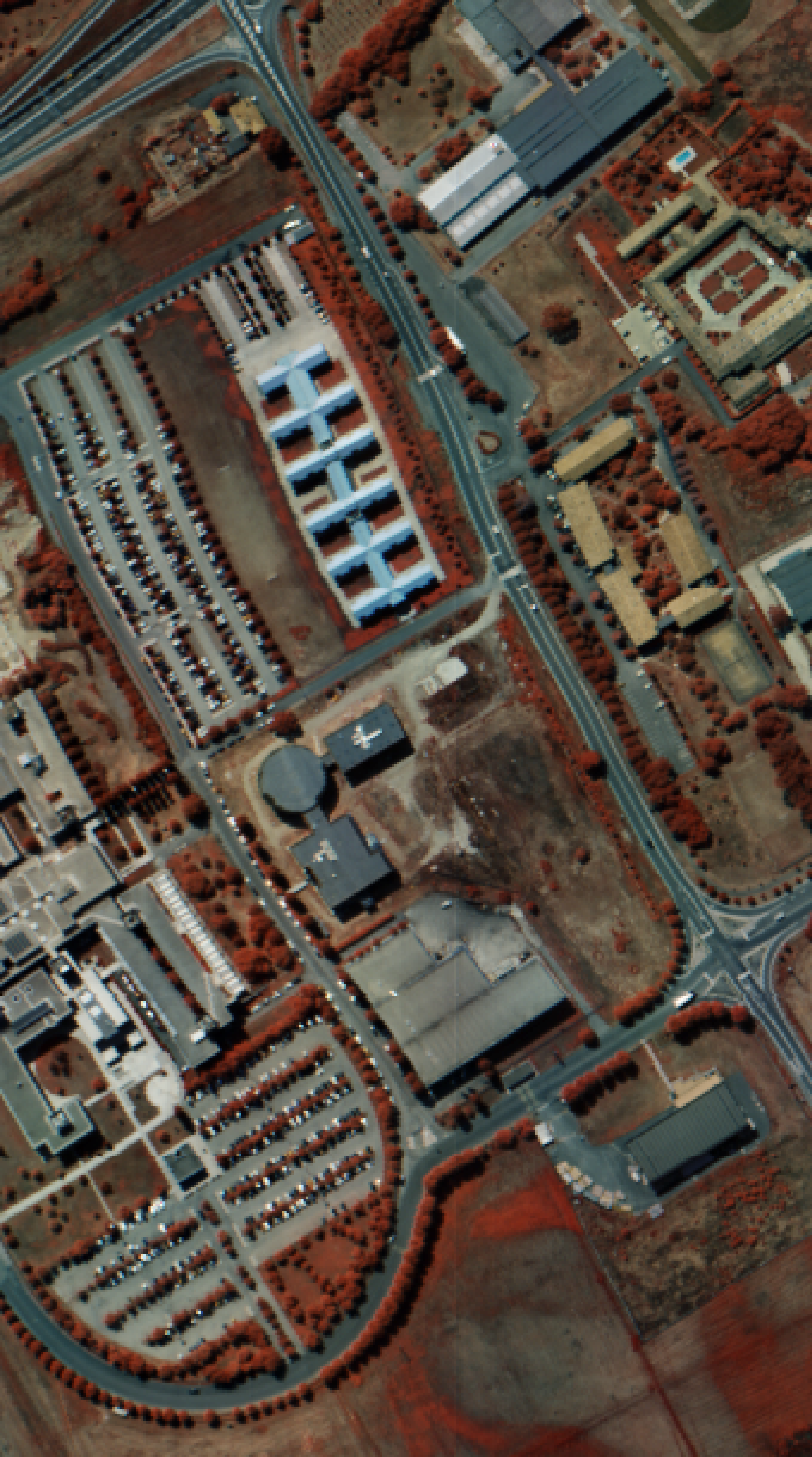} & \includegraphics[width = 0.15\textwidth]{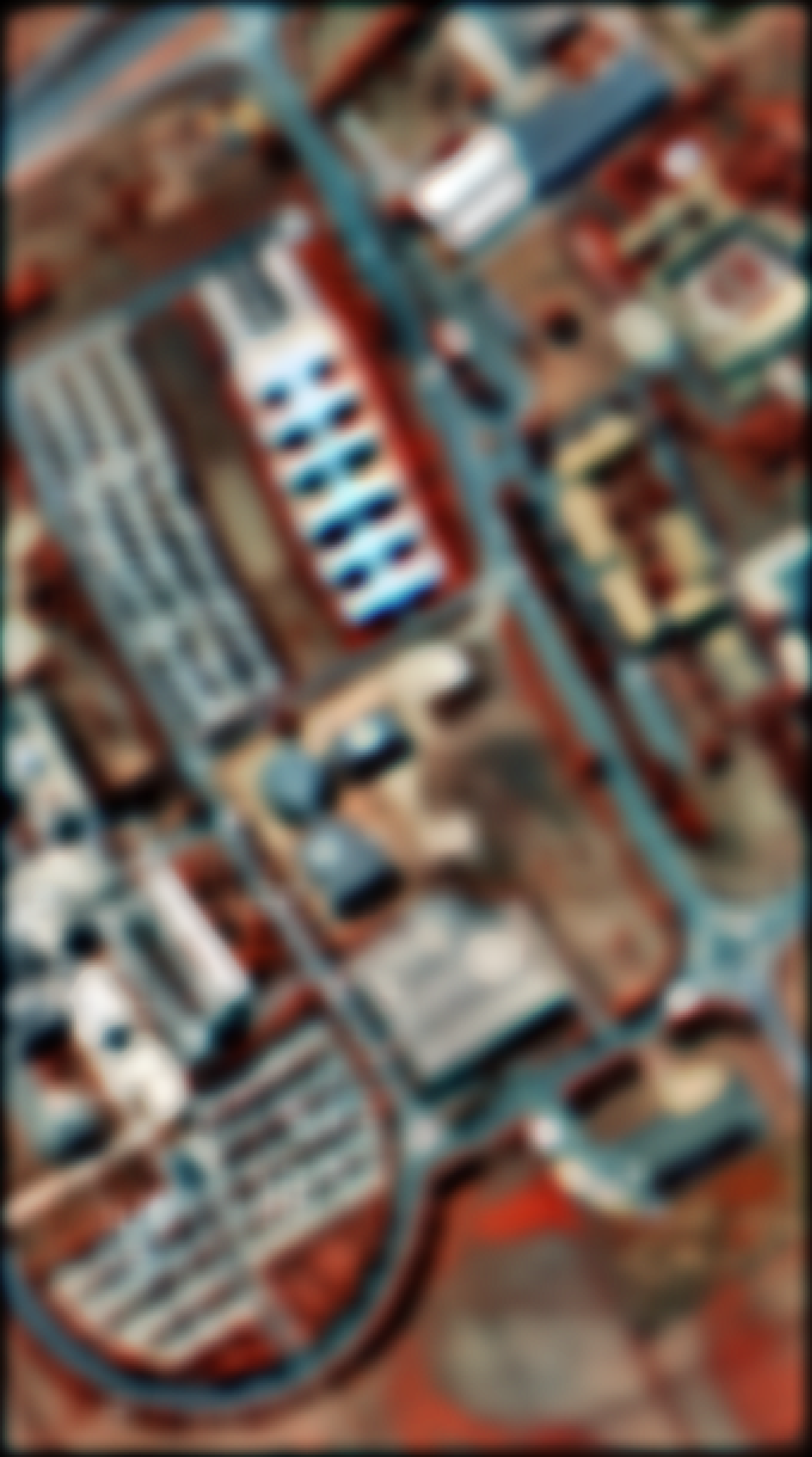} & \includegraphics[width = 0.15\textwidth]{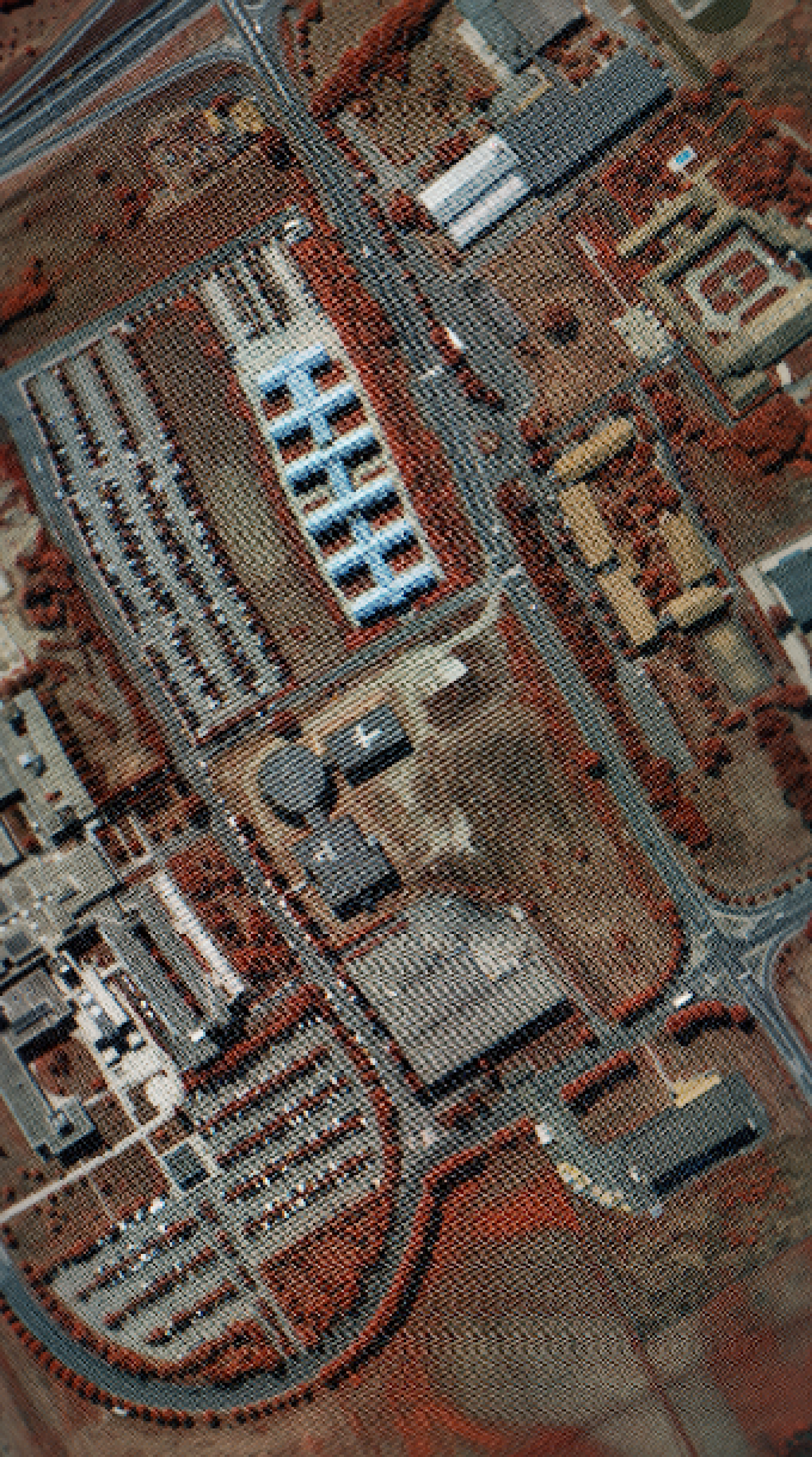} & \includegraphics[width = 0.15\textwidth]{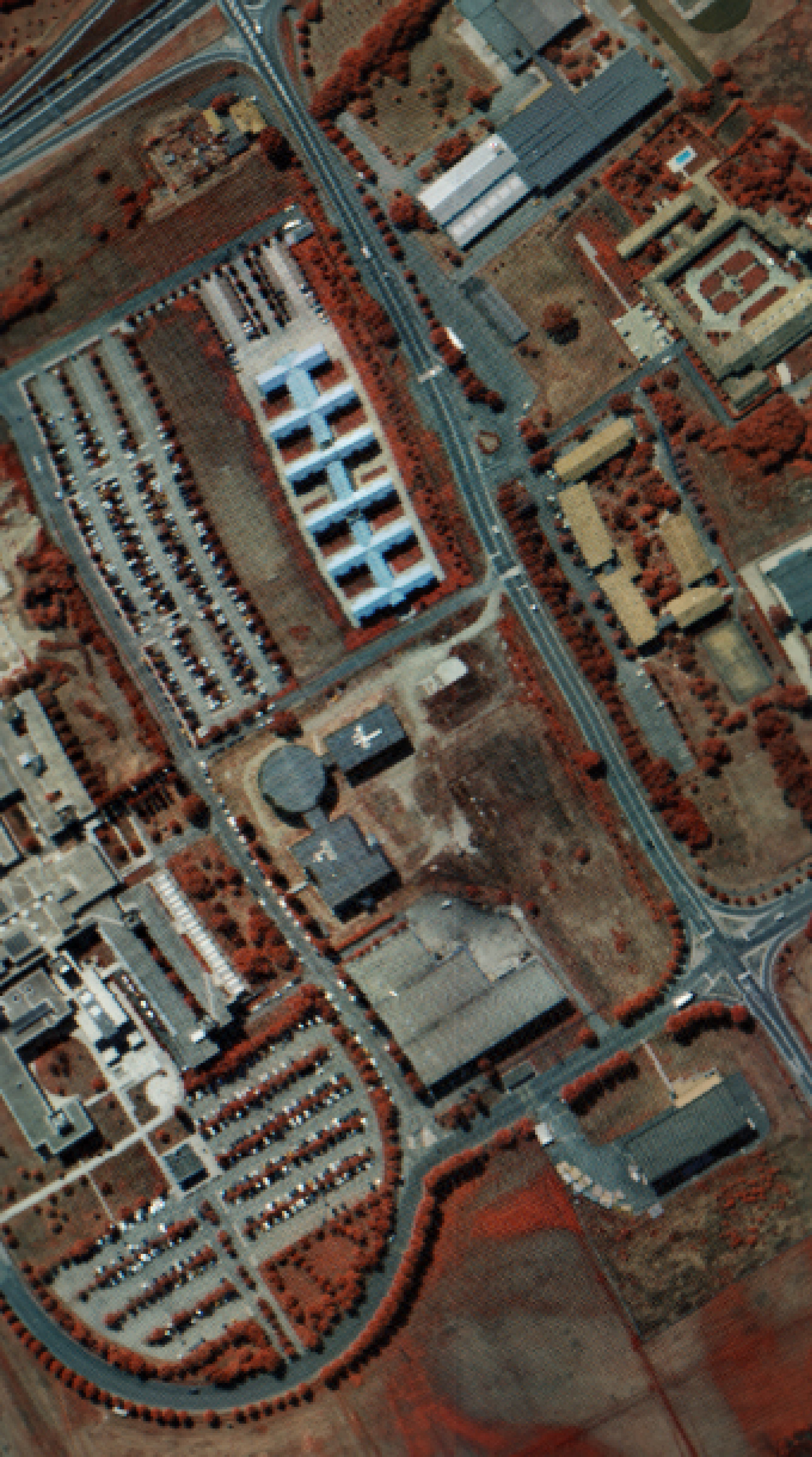} & \includegraphics[width = 0.15\textwidth]{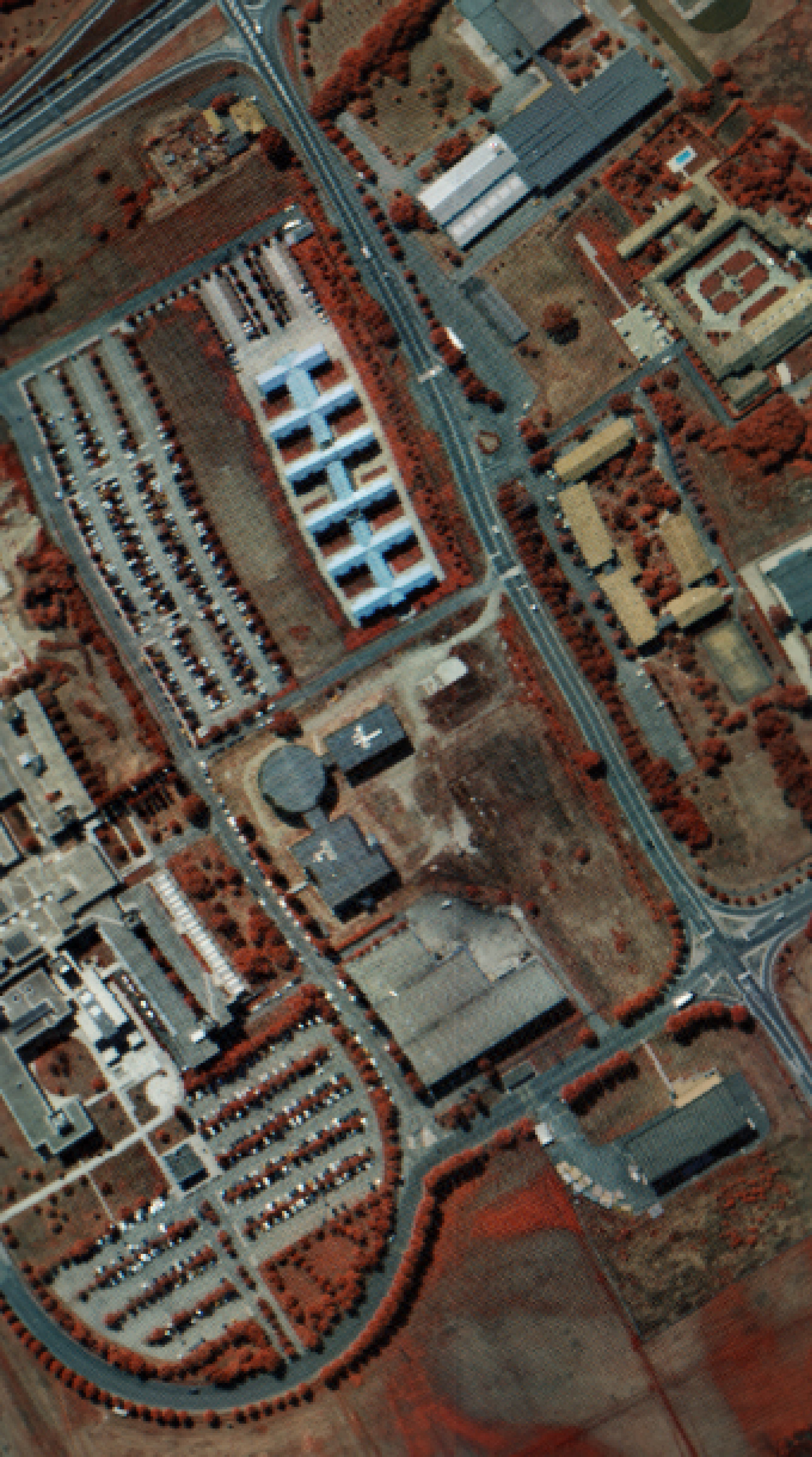}\\\hline

Salinas & \includegraphics[width = 0.15\textwidth]{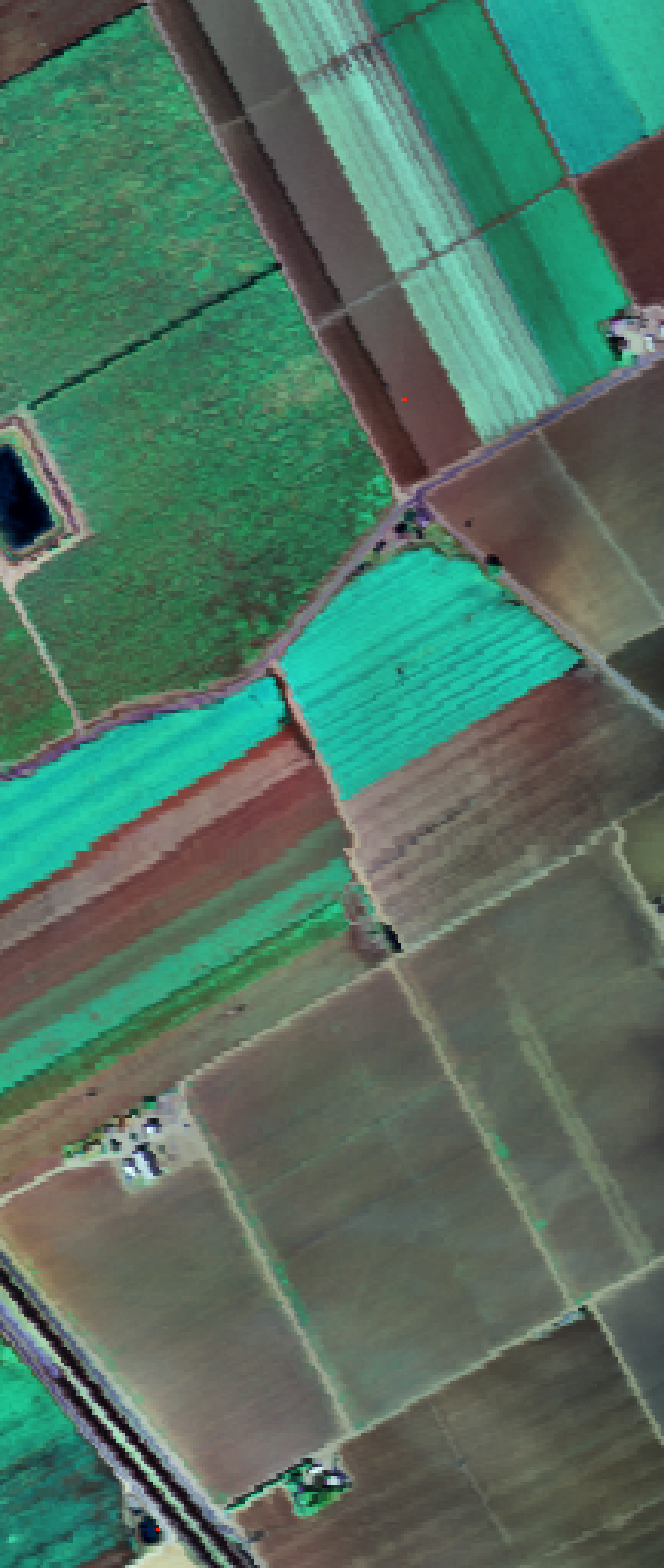} & \includegraphics[width = 0.15\textwidth]{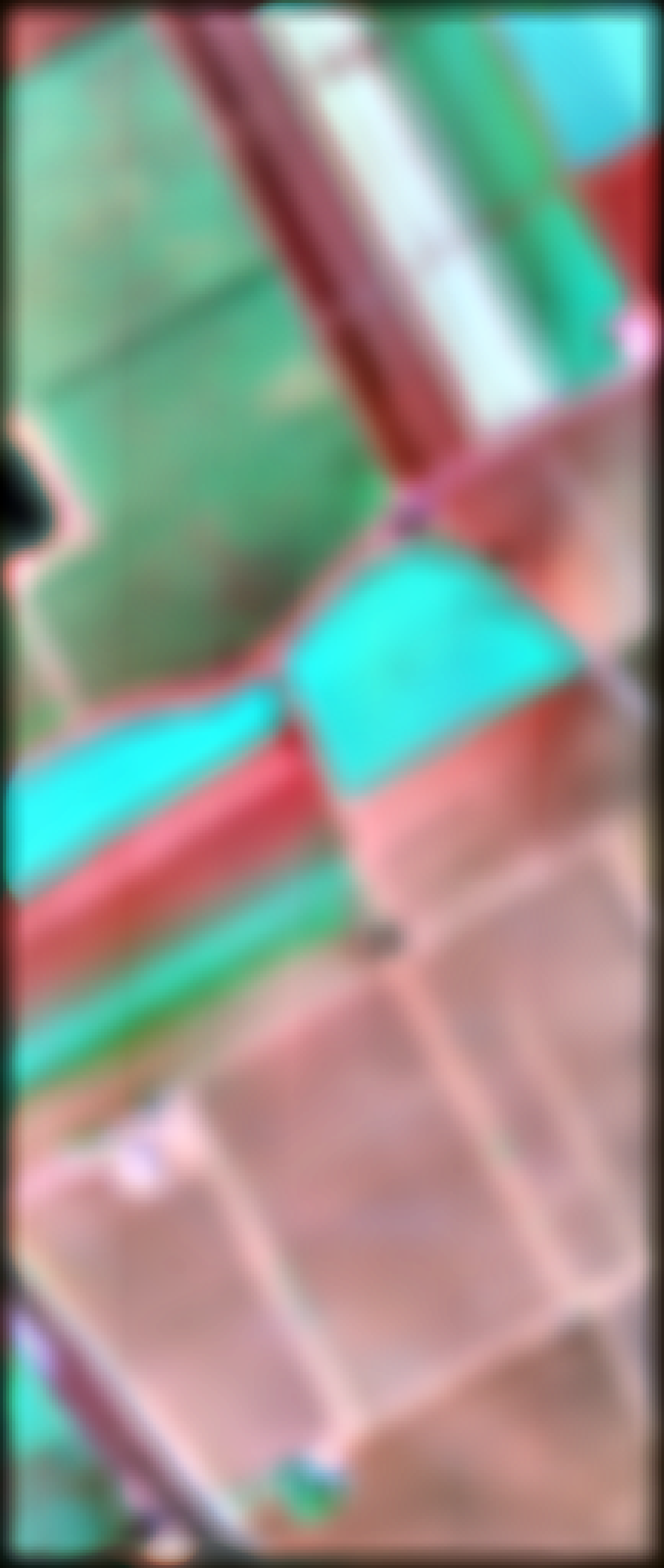} & \includegraphics[width = 0.15\textwidth]{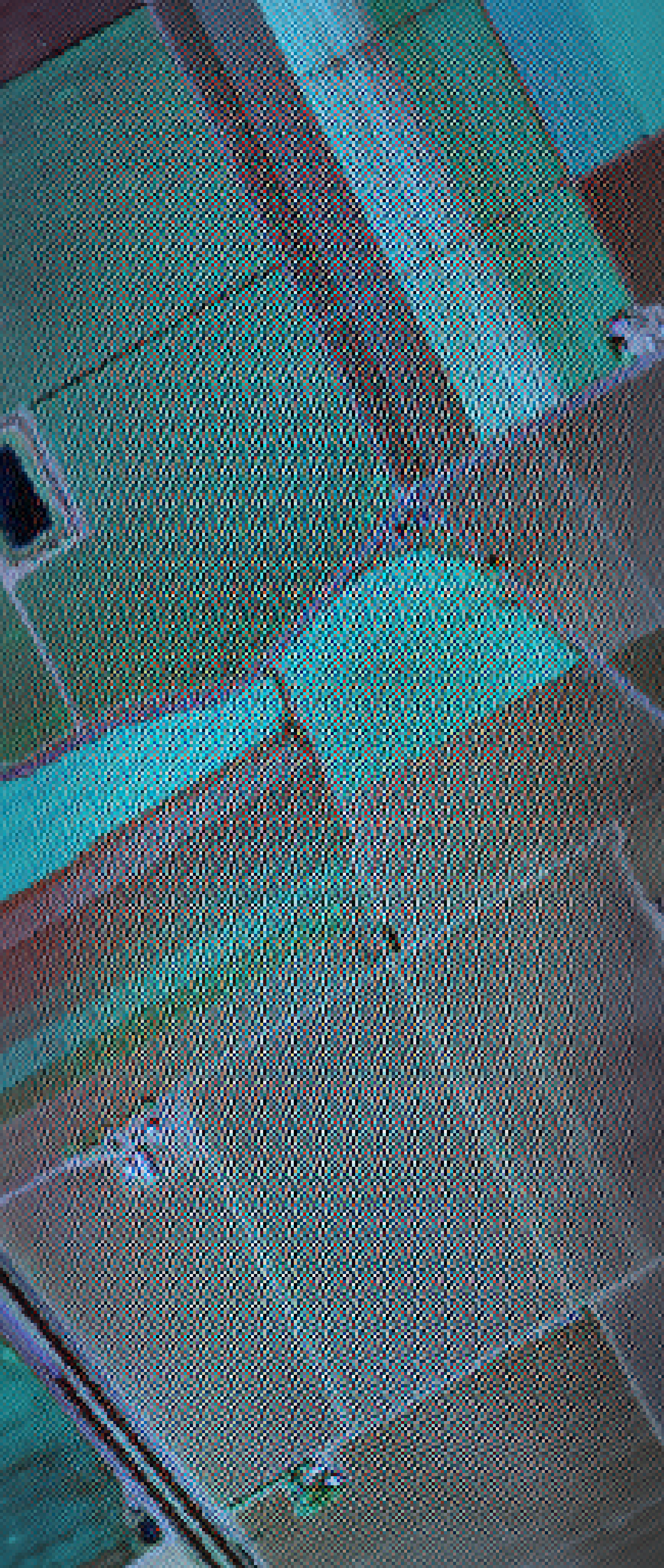} & \includegraphics[width = 0.15\textwidth]{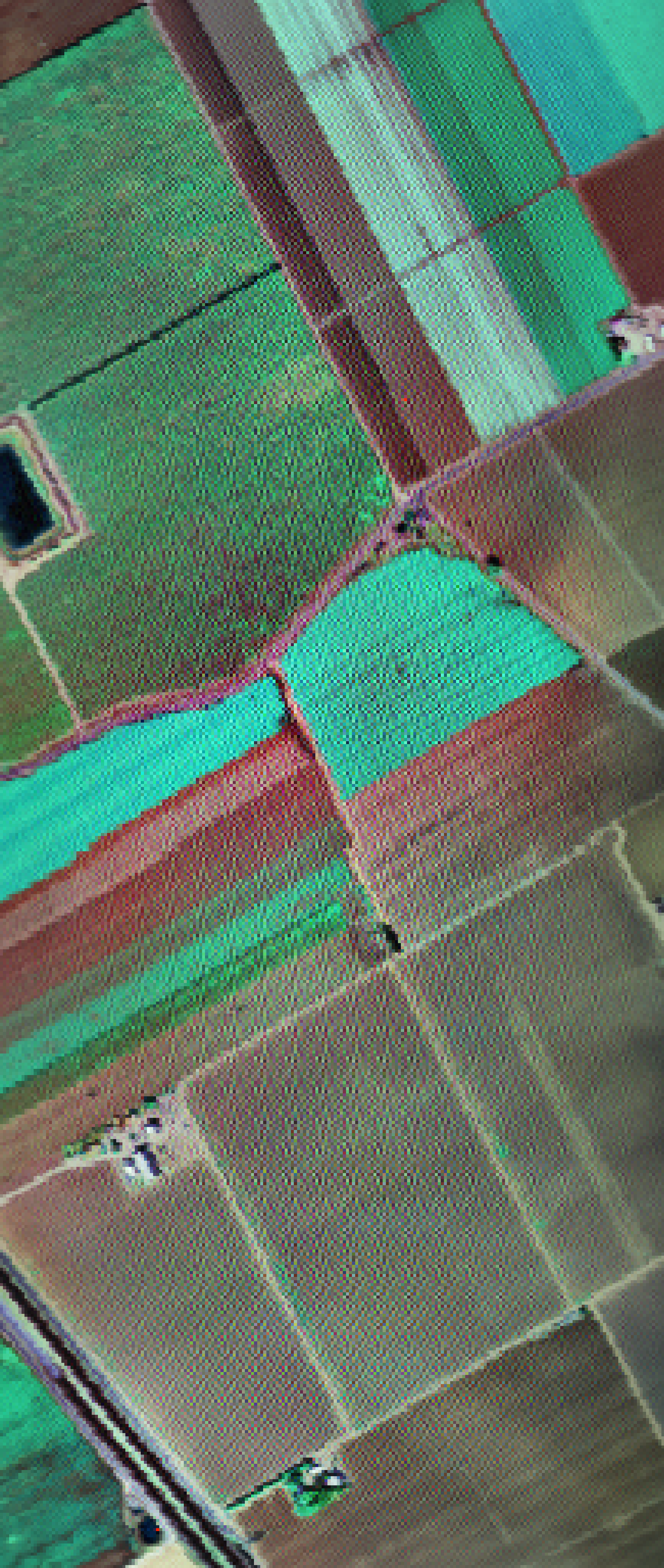} & \includegraphics[width = 0.15\textwidth]{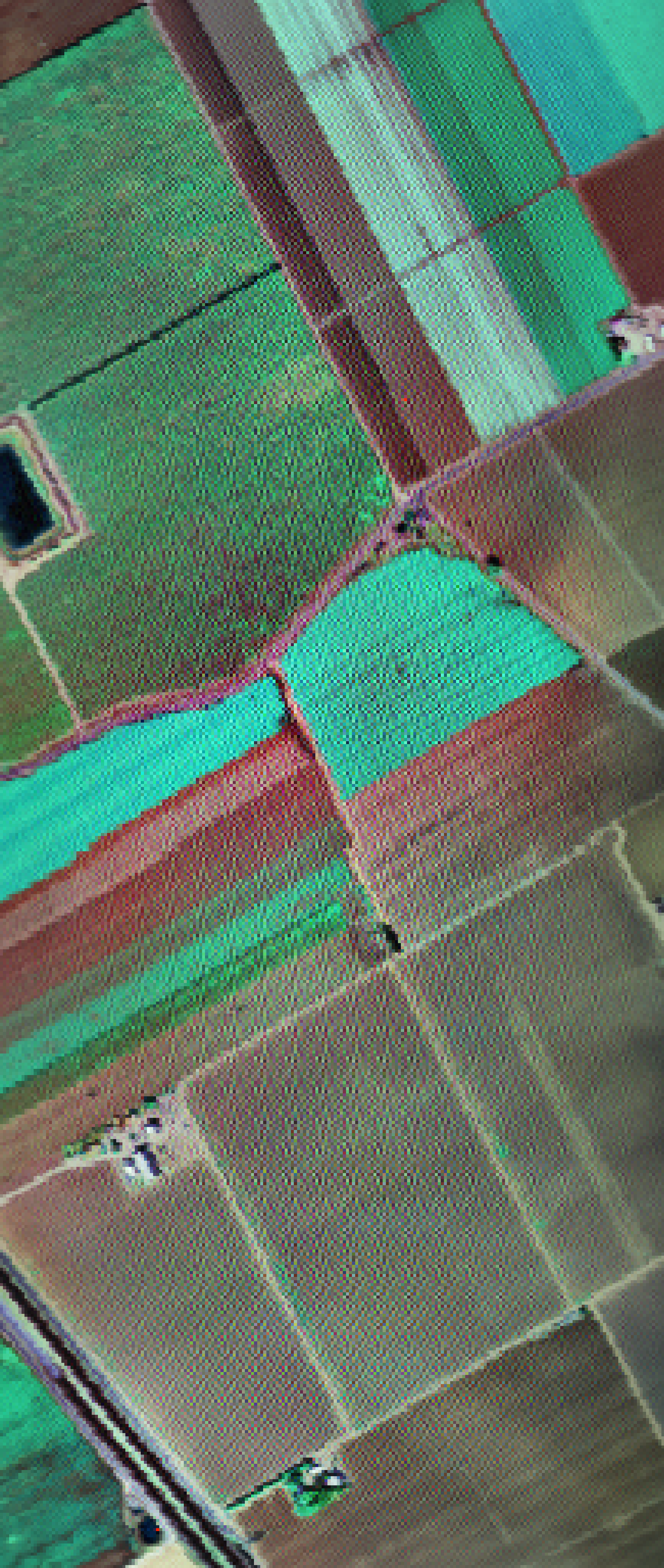}\\\hline

Botswana & \includegraphics[scale = 0.385]{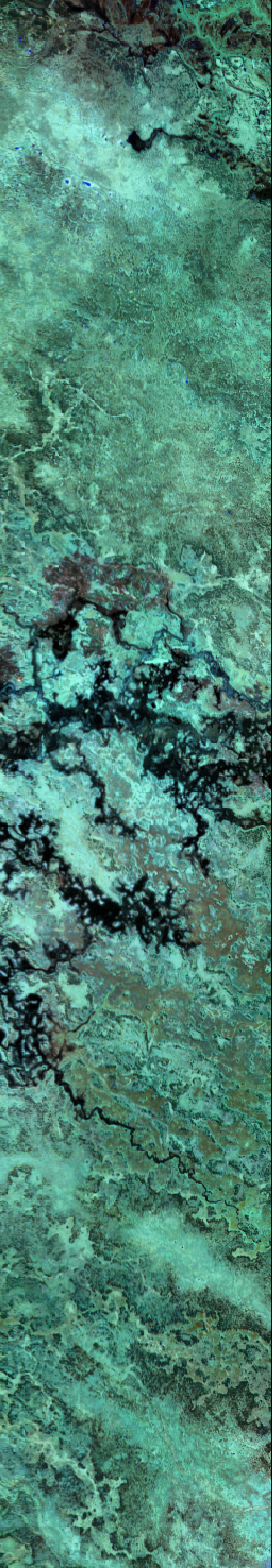} & \includegraphics[scale = 0.385]{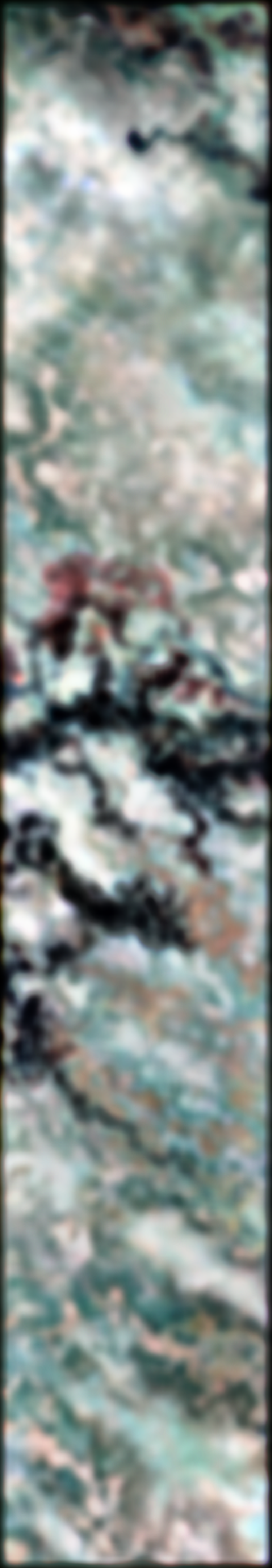} & \includegraphics[scale = 0.385]{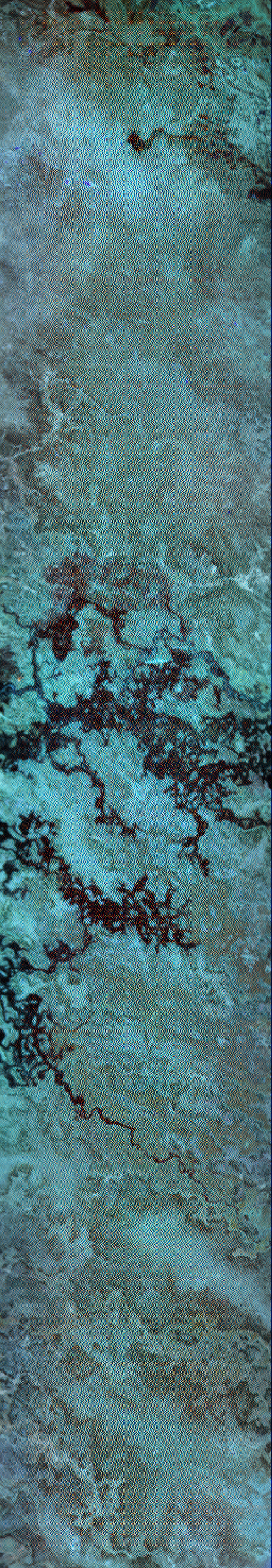} & \includegraphics[scale = 0.385]{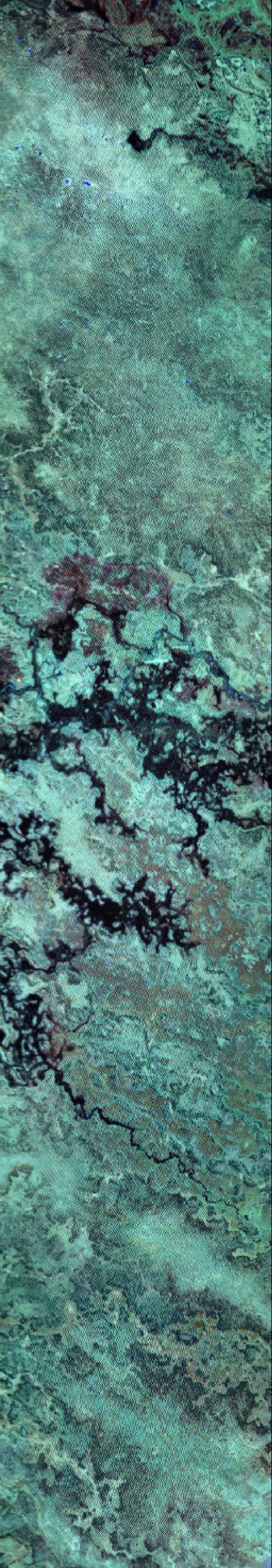} & \includegraphics[scale = 0.385]{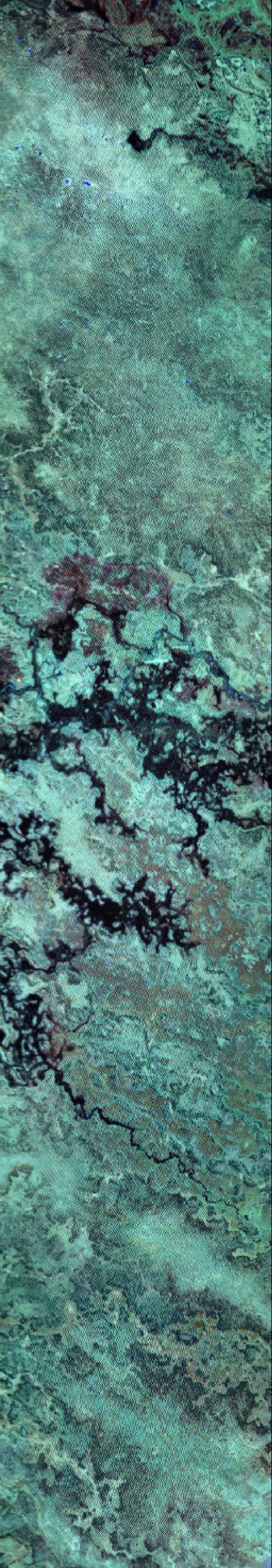}\\\hline

\end{tabular}}
\end{table*}

\section{Conclusions and Future Directions}\label{sec:conclusion}
The proposed wavelet-inspired $w$-product offers an efficient 
tensor multiplication framework, achieving linear transform complexity while retaining essential algebraic properties and enabling fast tensor decompositions. This novel approach not only facilitates fast tensor decompositions but also preserves multidimensional structures through a multiresolution representation that utilized second-generation wavelets. The Moore-Penrose inverse of the tensors is developed using the $w$-product, and its key properties are established. The supporting numerical examples validate the theoretical results. In low-rank hyperspectral image reconstruction, ``sp-$w$-svd'' achieved a speedup of up to $92.21$ times compared to ``$t$-svd'', with comparable PSNR and SSIM performance. This significant acceleration, coupled with impressive reconstruction fidelity, makes the $w$-product highly suitable for large-scale hyperspectral data processing. Further analysis on the hyperspectral image deblurring, ``sp-$w$-product'' reconstruction attains a speedup of up to $27.88$ times compared to ``$t$-product'' reconstruction with improved PSNR and SSIM values. The proposed $w$-product and sp-$w$-product decompositions achieved exponential increase in speedups as the level of decomposition increases compared to the $t$-product decompositions. In general, the proposed framework not only reduces computational cost but also provides a robust tool for advanced tensor-based hyperspectral image analysis.

Future work focuses on customizing wavelet operators for various data, extending the $w$-product to higher-order tensors and machine learning applications, refining theoretical bounds, and enabling real-time implementations on parallel hardware for large-scale multidimensional data analysis.


\section*{Declarations}
\subsection*{Ethical approval}
Not applicable.
\vspace{-0.4cm}
\subsection*{Availability of supporting data} 
The hyperspectral image data set used in this study is publicly available on Kaggle~\cite{dataset}, and Remote Sensing Data is provided by Acerca de Grupo de Inteligencia Computacional (GIC)~\cite{dataset2}.
\vspace{-0.4cm}
\subsection*{Competing interests}
The author declares that they have no conflicts of interest.
\vspace{-0.4cm}
\subsection*{Funding}
The second author is thankful to the ANRF, India, Grant No. EEQ/2022/001065.


\bibliographystyle{ieeetr}
\bibliography{bibliography}

\begin{thebibliography}{10}

\bibitem{panagakis2021tensor}
Y.~Panagakis, J.~Kossaifi, G.~G. Chrysos, J.~Oldfield, M.~A. Nicolaou, A.~Anandkumar, and S.~Zafeiriou, ``Tensor methods in computer vision and deep learning,'' {\em Proceedings of the IEEE}, vol.~109, no.~5, pp.~863--890, 2021.

\bibitem{fan2017hyperspectral}
H.~Fan, Y.~Chen, Y.~Guo, H.~Zhang, and G.~Kuang, ``Hyperspectral image restoration using low-rank tensor recovery,'' {\em IEEE Journal of Selected Topics in Applied Earth Observations and Remote Sensing}, vol.~10, no.~10, pp.~4589--4604, 2017.

\bibitem{qi2017tensor}
L.~Qi and Z.~Luo, {\em Tensor analysis: spectral theory and special tensors}.
\newblock SIAM, 2017.

\bibitem{kilmer2011}
M.~E. Kilmer and C.~D. Martin, ``Factorization strategies for third-order tensors,'' {\em Linear Algebra Appl.}, vol.~435, no.~3, pp.~641--658, 2011.

\bibitem{Miao2020Gene}
Y.~Miao, L.~Qi, and Y.~Wei, ``Generalized tensor function via the tensor singular value decomposition based on the {T}-product,'' {\em Linear Algebra Appl.}, vol.~590, pp.~258--303, 2020.

\bibitem{Songphiproduct}
G.~Song, M.~K. Ng, and X.~Zhang, ``Robust tensor completion using transformed tensor singular value decomposition,'' {\em Numer. Linear Algebra Appl.}, vol.~27, no.~3, pp.~e2299, 27, 2020.

\bibitem{kernfeld2015}
E.~Kernfeld, M.~Kilmer, and S.~Aeron, ``Tensor-tensor products with invertible linear transforms,'' {\em Linear Algebra Appl.}, vol.~485, pp.~545--570, 2015.

\bibitem{liu2024revisiting}
S.~Liu, X.-L. Zhao, J.~Leng, B.-Z. Li, J.-H. Yang, and X.~Chen, ``Revisiting high-order tensor singular value decomposition from basic element perspective,'' {\em IEEE Transactions on Signal Processing}, 2024.

\bibitem{cao2020enhanced}
X.~Cao, J.~Yao, X.~Fu, H.~Bi, and D.~Hong, ``An enhanced 3-d discrete wavelet transform for hyperspectral image classification,'' {\em IEEE Geoscience and Remote Sensing Letters}, vol.~18, no.~6, pp.~1104--1108, 2020.

\bibitem{liu2024bi}
Y.-Y. Liu, X.-L. Zhao, J.-Y. Xie, Z.~Xu, and G.~Vivone, ``Bi-level tensor decomposition for hyperspectral image restoration,'' in {\em IGARSS 2024-2024 IEEE International Geoscience and Remote Sensing Symposium}, pp.~7635--7638, IEEE, 2024.

\bibitem{xu2019nonlocal}
Y.~Xu, Z.~Wu, J.~Chanussot, and Z.~Wei, ``Nonlocal patch tensor sparse representation for hyperspectral image super-resolution,'' {\em IEEE Transactions on Image Processing}, vol.~28, no.~6, pp.~3034--3047, 2019.

\bibitem{kong2025image}
Z.~Kong, F.~Deng, and X.~Yang, ``Image denoising using green channel prior,'' {\em IEEE Transactions on Image Processing}, 2025.

\bibitem{appiah2024performance}
M.~K. Appiah, S.~K. Danuor, and A.~K. Bienibuor, ``Performance of continuous wavelet transform over fourier transform in features resolutions,'' {\em International Journal of Geosciences}, vol.~15, no.~2, pp.~87--105, 2024.

\bibitem{sharif2014comparative}
I.~Sharif and S.~Khare, ``Comparative analysis of haar and daubechies wavelet for hyper spectral image classification,'' {\em The International Archives of the Photogrammetry, Remote Sensing and Spatial Information Sciences}, vol.~40, pp.~937--941, 2014.

\bibitem{Swe98Thelift}
W.~Sweldens, ``The lifting scheme: a construction of second generation wavelets,'' {\em SIAM J. Math. Anal.}, vol.~29, no.~2, pp.~511--546, 1998.

\bibitem{luo2021redundant}
F.~Luo, Z.~Zou, J.~Liu, and Z.~Lin, ``Dimensionality reduction and classification of hyperspectral image via multistructure unified discriminative embedding,'' {\em IEEE Transactions on Geoscience and Remote Sensing}, vol.~60, pp.~1--16, 2021.

\bibitem{liao2023redundant}
J.~Liao, L.~Wang, and G.~Zhao, ``Hyperspectral image classification based on the gabor feature with correlation information,'' {\em Canadian Journal of Remote Sensing}, vol.~49, no.~1, p.~2246158, 2023.

\bibitem{dataset2}
M.~Graña, M.~A. Veganzons, B.~Ayerdi, and U.~d. P. V.~U. Grupo~de Inteligencia~Computacional, ``Hyperspectral remote sensing scenes.'' \url{https://www.ehu.eus/ccwintco/index.php?title=Hyperspectral_Remote_Sensing_Scenes}.
\newblock Accessed: 2025-11-16.

\bibitem{dataset}
Harcus, ``Near infrared hyperspectral image dataset.'' \url{https://www.kaggle.com/datasets/hacarus/near-infrared-hyperspectral-image}, 2020.
\newblock Accessed: 2025-11-02.

\end{thebibliography}

\end{document}